\documentclass[a4paper,11pt]{article}

\usepackage{xcolor}
\newcommand{\todo}[1]{}
\newcounter{ContadorComentario}
\usepackage{xcolor}

\usepackage[ruled,vlined,linesnumbered]{algorithm2e}
\usepackage{amsmath}
\usepackage{amsfonts}
\usepackage{amssymb}
\usepackage{amsthm}
\usepackage[affil-it]{authblk}
\usepackage{comment}
\usepackage{complexity}
\usepackage{enumerate}
\usepackage{fullpage}
\usepackage{graphicx,url}
\usepackage{hyperref}
\usepackage[utf8]{inputenc}
\usepackage{lineno}
\usepackage{rotating}
\usepackage{setspace}
\usepackage{tikz}

\newtheorem{theorem}{Theorem}
\newtheorem{corollary}[theorem]{Corollary}

\newtheorem{lemma}[theorem]{Lemma}

\newtheorem{claim}[theorem]{Claim}
\newtheorem{remark}[theorem]{Remark}

\newenvironment{proofClaim}[1]{\noindent {\em Proof of Claim~$\ref{#1}$.}}{\hfill$\blacksquare$\bigskip}

\newcommand{\sep}{\mathcal S}

\renewcommand{\top}[2]{\overset{\frown}{{\cal #1}^{#2}}}

\title{On restricted completions of chordal and trivially perfect graphs}

\author{Mitre C. Dourado$^{a,}$\thanks{Partially supported by Conselho Nacional de Desenvolvimento Científico e Tecnológico, Brazil, Grant number 305404/2020-2.}, Luciano N. Grippo$^{b,}$\thanks{Partially supported by ANPCyT PICT 2017-1315 and Programa Regional MATHAMSUD MATH190013.} and Mario Valencia-Pabon$^c$}

\begin{document}

\onehalfspace
\maketitle

\begin{center}
{\small
$^a$Instituto de Computação, Universidade Federal do Rio de Janeiro, \\ Rio de Janeiro, Brazil. E-mail:mitre@ic.ufrj.br\\
$^b$Instituto de Ciencias, Universidad Nacional de General Sarmiento, Argentina. E-mail:lgrippo@campus.ungs.edu.ar \\
$^c$LIPN, Universit\'{e} Paris-13, Sorbonne Paris Cit\'{e}, CNRS UMR7030, Villataneuse, France. \\E-mail:valencia@lipn.univ-paris13.fr
}
\end{center}

\begin{abstract}
Let $G$ be a graph having a vertex $v$ such that $H = G - v$ is a trivially perfect graph. We give a polynomial-time algorithm for the problem of deciding whether it is possible to add at most $k$ edges to $G$ to obtain a trivially perfect graph. This is a slight variation of the well-studied {\sc Edge Completion}, also known as {\sc Minimum Fill-In}, problem. We also show that if $H$ is a chordal graph, then the problem of deciding whether it is possible to add at most $k$ edges to $G$ to obtain a chordal graph is \NP-complete.
\end{abstract}

\section{Introduction}

\todo{incluir CNPq.}

We consider finite, simple, and undirected graphs. A {\em chordal graph} is a graph with no induced cycles of length at least four. A graph $G$ is {\em trivially perfect} if for every induced subgraph $G'$ of $G$, the cardinality of a maximum independent set of $G'$ is equal to the number of maximal cliques of $G'$. A {\em split graph} is a graph whose vertex set can be partitioned into an independent set and a clique. Given as input a graph $G$ and an integer $k$, the problem of deciding whether it is possible to add at most $k$ edges to $G$ to obtain a graph belonging to a graph class $\Pi$ can be formalized as follows.

\bigskip
\noindent {\sc $\Pi$ completion}\\
\begin{tabular}{lp{14cm}}
	Input: & A graph $G$ and an integer $k$.\\
	Question: & Is there a graph $H \in \Pi$ such that $V(H) = V(G)$, $E(G) \subseteq E(H)$ and $|E(H)| - |E(G)| \le k$?
\end{tabular}
\bigskip

The {\sc Chordal completion} problem was proved to be \NP-complete in~\cite{Y1981}. There are polynomial-time approximation algorithms~\cite{NSS2000}, exponential-time exact algorithms~\cite{FKT2004,FV2008}, and parameterized algorithms~\cite{BHV2011,C1996,DFPV2015,KST1994} for this problem.
The {\sc Trivially perfect completion} problem is also known to be \NP-complete~\cite{NG2013} and a polynomial kernel for this problem was presented in~\cite{DP2018}. We consider the following slight variation of this problem.

\bigskip
\noindent {\sc $\Pi$ $p$-completion}\\
\begin{tabular}{lp{14cm}}
	Input: & A graph $G$ having a set $T \subseteq V(G)$ such that $|T| \le p$ and $G - T \in \Pi$ and an integer $k$.\\
	Question: & Is there a graph $H \in \Pi$ such that $V(H) = V(G)$, $E(G) \subseteq E(H)$ and $|E(H)| - |E(G)| \le k$?
\end{tabular}
\bigskip

Given graphs $G$ and $H$ and a graph class $\Pi$, we say that $H$ is a {\em $\Pi$ completion of $G$} if  $G$ is a spanning subgraph of $H$ and $H \in \Pi$. Every edge of $E(H) \setminus E(G)$ is called a {\em fill edge} and we denote $fill(H) = |E(H)| - |E(G)|$. Note that for any positive integer $p$, the {\sc $\Pi$ $p$-completion} problem is a restriction of the {\sc $\Pi$ completion} problem in the sense that an input for {\sc $\Pi$ $p$-completion} is also an input for {\sc $\Pi$ completion}.

We remark that the input of the {\sc $\Pi$ $p$-completion} problem is formed only by a graph $G$ and an integer $k$, i.e., to know more than one set $T \subset V(G)$ such that $|T| \le p$ and $G - T \in \Pi$ can help in the search for a minimum $\Pi$ completion of $G$, since each such set has a different structure that makes the graph $G$ not to belong to $\Pi$, however, this does not change the answer because the question of the problem asks only about $G$ and $k$. We also remark that the case $p=1$ was proposed in~\cite{LMP10} with another formulation and this is the problem we investigate in this work. In fact, the definition given in~\cite{LMP10} has as input a graph $G \in $ {\sc Cograph} and asks for a minimum {\sc Cograph} completion of the graph obtained by adding one vertex $v$ to $G$ (jointly with some edges incident to $v$), which, as observed above, is equivalent to the {\sc {\sc Cograph} $1$-completion} problem.

The text is organized as follows. In Section~\ref{sec:chordal}, we show that {\sc Chordal $1$-completion} is \NP-complete. We also prove that given a graph $G$ having a vertex $v$ such that $G - v$ is a split graph and an integer $k$, it is \NP-complete to decide whether $G$ has chordal completion with at most $k$ edges. In Section~\ref{sec:TP}, we present an algorithm for solving the {\sc Trivially perfect $1$-completion} problem in polynomial time.

We conclude this section introducing some notations. For a natural number $k$, we denote $\{1, \ldots, k\}$ by $[k]$.
We use $H \subset G$ to say that $H$ is a proper subgraph of $G$.
The neighborhood of $v \in V(G)$ is denoted by $N_G(v)$.

\section{Chordal and split graphs} \label{sec:chordal}

Consider a graph $G$. A set of vertices $S$ of $G$ is said to be a \emph{vertex separator} if there exist two vertices $v$ and $w$ of $G$ such that $v$ and $w$ are in distinct connected component of $G-S$. A minimal separator set is called \emph{minimal} if it is minimal under inclusion. We will denote by $\sep (G)$ the set of all minimal vertex separators of $G$.

\begin{theorem} {\em \cite{Gol04}}\label{thm:chordal-charact}
Let G be a graph. The graph $G$ is chordal if and only if every minimal vertex separator
set is a clique.
\end{theorem}

\begin{remark}\label{rmk:minimal-separator}
Let $G$ be a graph. If $H$ is a $\Pi$ completion of $G$ and $S$ is a vertex separator of $H$, then $S$ is a vertex separator of $G$.
\end{remark}

\begin{lemma}\label{lem:simplicialvetex-in-completions}
Let $G$ be a graph and let $H$ be a minimum chordal completion of $G$. Then, for every clique $C$ of $H$ properly contained in $N_H(v)$ for $v\in V(G)$, there exists a vertex $x\in V(H) \setminus C$ such that $vx \in E(G)$.
\end{lemma}

\begin{proof}
Suppose, by contradiction, that for every vertex $x\in N_G(v) \setminus C$, $vx$ is a fill edge. Let us call $F'$ to the set of fill edges $vx$ with $x\in N_G(v) \setminus C$. Consider the graph $H'=H-F'$. Notice that $v$ is a simplicial vertex of $H'$ and $F' \neq \varnothing$. In addition, $H'-v=H-v$ and thus $H'$ is a chordal graph. Consequently, $H'$ is a chordal completion of $G$ with less fill edges than $H$, contradicting that $H$ is a minimum chordal completion of $G$. The contradiction arises from supposing that for every vertex $x\in N_G(v) \setminus C$, $vx$ is a fill edge. Therefore, $N_G(v)\cap (V(H) \setminus C) \neq \varnothing$.
\end{proof}

\begin{lemma}\label{lem:minimal-separator-completation}
Let $G$ be a graph having a vertex $v$ such that $G - v$ is a non-complete connected chordal graph, and let $H$ be a minimum chordal completion of $G$. If no clique separator $S \in\sep (G)$ is contained in $N_H(v)$, then $N_G(v) \in\sep (H)$.
\end{lemma}

\begin{proof}
Since $N_H(v)$ is a vertex separator of $H$, $H$ is a chordal graph, and $N_H(v)$ contains no minimal clique separator in $\sep(G)$, Theorem~\ref{thm:chordal-charact} implies that $N_H(v)$ is a minimal clique separator of $H$. So, by Lemma \ref{lem:simplicialvetex-in-completions}, $N_H(v)=N_v(G)$, which means that $N_v(G) \in\sep (H)$.
\end{proof}

\begin{theorem}{\em \cite{Y1981}}\label{thm:np-chordal-completion-co-bipartite-graphs}
Given a co-bipartite graph $G$ and an integer $k$, it is \NP-complete to decide whether there exists a chordal completion of $G$ with at most $k$ fill edges.
\end{theorem}

Given a co-bipartite graph $F$ with $(A,B)$ a bipartition of $\overline{F}$ and an integer $k$, set $K=k+\frac {1}{2} |A| (|A|-1)$ and define $F^*$ as the graph with vertex set $V(F)\cup\{c_1, \ldots ,c_{K+1}\}$ and $E(F^*)=(E(F) \setminus E(F[A]))\cup\{c_iu : i \in [K+1] \mbox{ and } u\in V(F) \setminus \{c_i\}\}$. In other words, $F^*$ is the graph obtained from $F$ by deleting all edges with both endpoints in $A$ and adding $K$ universal vertex to the resulting graph. Notice that $F^*$ is a split graph and thus a chordal graph. Denote by $G$ the graph with vertex set $V(G) = V(F) \cup \{v\}$ where $v$ is a new vertex and edge set $E(G) = E(F^*) \cup \{va : a \in A\}$. Write $C=\{c_1,\ldots,c_{K+1}\}$.

\begin{lemma}\label{lem:np-complete}
Let $F$ be a co-bipartite graph with $(A,B)$ a bipartition of $\overline{F}$, let $k$ be an integer and let $G$ be constructed from $F$ as above. Then, $F$ has a minimum chordal completion with at most $k$ fill edges if and only if $G$ has a minimum chordal completion with at most $K$ edges.
\end{lemma}

\begin{proof}
Suppose that $F$ has a minimum chordal completion with at most $k$ fill edges. Let us call $M$ to that set of fill edges. Consequently, a chordal completion $H$ of $G$ with $K$ edges can be obtained by adding all the edges with both endpoints in $A$ plus the edges of the set $M$. In order to show that $H$ is indeed a chordal completion of $G$, it suffices to observe that for any vertex $u \in V(H) \setminus V(F)$, $N_H(u)$ is a clique or $u$ is a universal vertex of $H$, i.e., $u$ does not belong to an induced $C_p$ for $p \ge 4$.

Conversely, suppose that $G$ has a minimum chordal completion $H$ with at most $K$ edges. Let us call $M$ to that set of fill edges. Notice that every minimal clique separator of $G - v$ contains the clique $C$.

We claim that for any $a,a' \in A$, it holds that $aa' \in E(H)$. Then, suppose the contrary and let $a,a' \in A$ such that $aa' \not\in E(H)$. Since $|M| \le K$, there is $c_i \in C$ such that $vc_i \not\in E(H)$, which implies that $vac_ia'$ is an induced $C_4$ of the chordal graph $H$, a contradiction. Hence, for any $a,a' \in A$, it holds that $aa' \in E(H)$.

Since $H$ has at most $K$ fill edges, $N_H(v)$ contains no minimal clique separator of $\sep(G)$. Therefore, by Lemma~\ref{lem:minimal-separator-completation}, $N_G(v)$ is a minimal separator of $H$ and thus, by Theorem \ref{thm:chordal-charact}, $N_G(v)$ is a clique in $H$. By Lemma \ref{lem:simplicialvetex-in-completions}, $N_H(v)=N_G(v)$. So, there is no fill edge having an endpoint in $C \cup \{v\}$. Since every edge with both endpoints in $A$ is a fill edge, the amount of edges with both endpoints in $A$ is $\frac{1}{2}|A|(|A|-1)$ and $C$ is a clique formed by universal vertices of $H-v$, it follows that the remaining $k' \le k$ edges have one endpoint in $A$ and the other in $B$. Therefore, adding these $k'$ edges to $F$ gives a chordal completion of $F$ with at most $k$ edges. 
\end{proof}

By combining Theorem \ref{thm:np-chordal-completion-co-bipartite-graphs} and Lemma \ref{lem:np-complete}, we obtain the following result.

\begin{theorem}
The {\sc Chordal $1$-completion} problem is \NP-complete. In addition, given a graph $G$ having a vertex $v$ such that $G - v$ is a split graph and an integer $k$, it is \NP-complete to decide whether $G$ has a chordal completion with at most $k$ edges.
\end{theorem}

\section{Trivially perfect graphs} \label{sec:TP}

Trivially perfect graphs were introduced by Wolk~\cite{Wolk62} as {\em comparability graphs of trees}. Golumbic~\cite{Gol78} called them {\em trivially perfect} graphs because it is easy to show that they are perfect from a characterization of perfect graphs, and showed that these graphs are precisely those graphs that are $\{P_4,C_4\}$-free. This characterization implies that the class of trivially perfect graphs is the subclass of cographs which are also chordal graphs or interval graphs.
In this section, we present a polynomial-time algorithm for finding a minimum trivially perfect completion of a graph $G$ having a vertex $v$ such that $G - v$ is trivially perfect.

For short, we will use ``TP completion" standing for trivially perfect completion and write ``$(G,F,v) \in$ {\sc 1TP}" if $G$ is connected a graph, $v \in V(G)$ and $F = G - v$ is a trivially perfect graph.
We denote the number of connected components of a graph $G$ by $\omega(G)$. A {\em rooted forest} is a set of rooted trees. It is known~\cite{YCC96} that there is a bijection $\alpha$, up to isomorphism, among the trivially perfect graphs and the rooted forests. In the remaining of the text, given a trivially perfect graph $G$, we will reserve the corresponding calligraphic letter to represent $\alpha(G)$, i.e., ${\cal G} = \alpha(G)$ and $G = \alpha^{-1}({\cal G})$.
We adopt the notation $|G| = |E(G)|$ and $|{\cal G}| = |V(G)|$.
If $G$ is connected, then we also denote the root of ${\cal G}$ by $g$. Note that $g$ is a universal vertex of $G$ and since $G$ can have other universal vertices, $g$ can have descendants that are also universal vertices of $G$.
We denote the trees of a rooted forest using the same letter with an underlined superscript indicating its position in a non-increasing ordering of their number of vertices, i.e., given a rooted forest ${\cal G}$,
one of the rooted trees of ${\cal G}$ with maximum number of vertices is chosen arbitrarily to be ${\cal G}^{\underline{1}}$, and 
for rooted trees ${\cal G}^{\underline{i}}$ and ${\cal G}^{\underline{j}}$ such that $|{\cal G}^{\underline{i}}| \ge |{\cal G}^{\underline{j}}|$ it holds if $i < j$. The root of ${\cal G}^{\underline{i}}$ is denoted by $g^{\underline{i}}$.

Let ${\cal T}$ be a rooted tree and let $u,w \in V({\cal T})$. We remark that the neighborhood of $u$ in $T$ is formed by the descendant and the ancestral vertices of $u$ in ${\cal T}$.
We denote ${\cal T}_u$ for the subtree of ${\cal T}$ rooted at $u$, $\pi(u)$ is the parent of $u$ in ${\cal T}$ and we can also use $T_u$ meaning $\alpha^{-1}({\cal T}_u)$.
If $u$ is ancestral of $w$, we denote the set of vertices that are descendants of $u$ and ancestral of $w$ in ${\cal T}$ by $[u,w]_{\cal T}$. We also use
$(u,w]_{\cal T} = [u,w]_{\cal T} \setminus \{u\}$,
$[u,w)_{\cal T} = [u,w]_{\cal T} \setminus \{w\}$ and
$(u,w)_{\cal T} = [u,w]_{\cal T} \setminus \{u,w\}$.
Every vertex of $[t,w)_{\cal T}$ is a {\em proper ancestral of $w$ in ${\cal T}$}.
The {\em level of $u$ in $V({\cal T})$} is 1 plus the distance of $u$ to the root of ${\cal T}$ and is denoted by $\ell_{\cal T}(u)$.
Observe that $|T| = \underset{{\cal T}^i \in {\cal T}}{\overset{v \in V({\cal T}^i)}{\sum}} (|{\cal T}^i_v|-1)$.

Every rooted tree ${\cal T}$ has also a special vertex called the {\em base of ${\cal T}$}, for which we reserve the letter $\overline{t}$. When not explicitly designated, the base of a tree is equal to its root. The base appears in the following definitions.

\begin{itemize}

\item For every $u,w \in V(T)$ such that $u$ is ancestral of $w$ and $w$ is ancestral of $\overline{t}$. If $w \ne \overline{t}$, then the subtree ${\cal T}_{u,w}$ is defined as ${\cal T}_u - {\cal T}_x$ such that $x$ is the child of $w$ that is ancestral of $\overline{t}$. The base of ${\cal T}_{u,w}$ is $w$. If $w = \overline{t}$, then we define ${\cal T}_{u,w} = {\cal T}_u$.

\item Given another rooted tree ${\cal G}$, denote by $\langle {\cal G}, {\cal T} \rangle = \langle {\cal G}_{g,\overline{g}}, {\cal T}_{t,\overline{t}} \rangle$ the rooted tree obtained by adding the edge $\overline{g}t$ with root $g$ and base $\overline{t}$.

\bigskip

For the following two definitions, let $k = \ell_{\cal T}(\overline{t})$ and for every $j \in [k]$, denote by $u_j$ the ancestral of $\overline{t}$ in ${\cal T}$ such that $\ell_{\cal T}(u_j) = j$.

\item For $1 \le i \le j \le k$, define the {\em average of ${\cal T}_{u_i,u_j}$} as $a({\cal T}_{u_i,u_j}) = \frac{\underset{q \in \{i, \ldots, j\}}{\sum} |{\cal T}_{u_q,u_q}|}{j-i+1}$. We will use $a({\cal T}) = a({\cal T}_{t,\overline{t}})$.

\item We say that ${\cal T}_{t,u_i}$ is the {\em leading subtree of ${\cal T}$} if $i \in [k]$ is the minimum number such that $a({\cal T}_{t,u_i}) \ge a({\cal T}_{t,u_j})$ for every $j \in [k]$.

\end{itemize}

Given a rooted tree ${\cal G}$, note that ${\cal G}_{g,\overline{g}} = {\cal G}$ and that if $\overline{g} = g$, then ${\cal G}_{g,g} = {\cal G}$. Note also that
if $u$ and $w$ are twins in $G$ and $u$ is ancestral of $w$ that is a proper ancestral of $\overline{g}$ in ${\cal G}$, then $V({\cal G}_{u,u}) = \{u\}$.
For $k \ge 3$ and rooted trees ${\cal T}^1, \ldots, {\cal T}^k$, we define $\langle {\cal T}^1, \ldots, {\cal T}^k \rangle = \langle \langle  {\cal T}^1, \ldots, {\cal T}^{k-1} \rangle, {\cal T}^{k} \rangle$.
Note that every rooted tree ${\cal T}$ can be decomposed into
$\langle {\cal T}_{t,\pi(u_1)}, {\cal T}_{u_1,\pi(u_2)}, \ldots, {\cal T}_{u_k,\overline{t}} \rangle$ for any $k \le \ell_{\cal T}(\overline{t})$ where $t$ is ancestral of $u_1$, $u_i$ is ancestral of $u_{i+1}$ for $i \in [k-1]$ and $u_k$ is ancestral of $\overline{t}$.
Hence, given a possibly disconnected trivially perfect graph $F$ and a connected TP completion $H$ of $F$, there is a TP completion $G$ of $F$ with $\omega(G) = \omega(F)$ such tat ${\cal H} = \langle {\cal X}^1_{x^1,x_1}, \ldots, {\cal X}^k_{x^k,x_k} \rangle$ for some $k \ge \omega(F)$ where ${\cal X}^i$ is a tree of 
${\cal G} - \{{\cal X}^1_{x^1,x_1}, \ldots, {\cal X}^{i-1}_{x^{i-1},x_{i-1}}\}$ for every $i \in [k]$ and $x_i$ is ancestral of $\overline{x_i}$ in ${\cal X}^i$.
We say that ${\cal H}$ is a {\em merge of ${\cal G}$}; and if for $i \in [k]$, ${\cal X}^i_{x^i,x_i}$ is the leading subtree with maximum average among the trees of ${\cal G} - \{{\cal X}^1_{x^1,x_1}, \ldots, {\cal X}^{i-1}_{x^{i-1},x_{i-1}}\}$, then we say that ${\cal H}$ is a {\em leading merge of ${\cal G}$}.

\begin{lemma} \label{lem:twotrees}
Let ${\cal T}^1$ and ${\cal T}^2$ be rooted trees. For $j \in \{1,2\}$, let ${\cal T}^j$ be a rooted tree. If $a({{\cal T}^1}) \ge a({{\cal T}^2})$, then $|G_1| \le |G_2|$ where ${\cal G}_1 = \langle {\cal T}^1, {\cal T}^2 \rangle$ and ${\cal G}_2 = \langle {\cal T}^2, {\cal T}^1 \rangle$.
\end{lemma}

\begin{proof}
Write $\ell_1 = \ell_{{\cal T}^1}(\overline{t^1})$ and $\ell_2 = \ell_{{\cal T}^2}(\overline{t^2})$. By definition, $a({{\cal T}^1}) = \frac{|{\cal T}^1|}{\ell_1} \ge \frac{|{\cal T}^2|}{\ell_2} = a({{\cal T}^2})$. Then, $\ell_2 |{\cal T}^1| \ge \ell_1 |{\cal T}^2|$. It remains to note that $|G_1| = \ell_1 |{\cal T}^2| + |T^1| + |T^2|$ and that $|G_2| = \ell_2 |{\cal T}^1| + |T^1| + |T^2|$.
\end{proof}

\begin{lemma} \label{lem:average}
Let ${\cal F}$ be a family of rooted trees and let ${\cal H}$ be a merge of ${\cal F}$. If ${\cal H}$ is not a leading merge of ${\cal F}$, then there are vertices $u_1,u_2,u_3 \in V({\cal F})$ such that $u_1 \in [h,u_2)_{\cal H}$, $u_2 \in (u_1,u_3)_{\cal H}$, $u_3 \in (u_2,\overline{h}]_{\cal H}$ and $|H'| < |H|$ where
${\cal H}' = \langle {\cal H}_{h,\pi(u_1)}, {\cal H}_{u_2,u_3}, {\cal H}_{u_1,\pi(u_2)}, {\cal H} - {\cal H}_{h,u_3} \rangle$.
\end{lemma}

\begin{proof}
Suppose by contradiction that the result does not hold. We can assume that ${\cal F}$ is a family with minimum number of vertices for which the result does not hold. Denote the leading subtree of ${\cal F}$ with maximum average as ${\cal T}^1$. Then, $\overline{t^1}$ has an ancestral in ${\cal H}$ not belonging to ${\cal T}^1$.
Denote the set of ancestrals of $\overline{t^1}$ in ${\cal H}$ by $\{w_1, \ldots, w_p\} = [h,\overline{t^1}]_{\cal H}$ and denote by ${\cal G}^1, \ldots, {\cal G}^q$ the partition of 
${\cal H}_{h,\overline{t^1}}$ into subtrees such that for every $i \in [q]$,
${\cal G}^i = {\cal H}_{w_j,w_k}$ is such that

\begin{itemize}
	\item either
	$\{w_j, \ldots, w_k\} \subseteq V({\cal T}^1)$, 
	$j = 1$ or $w_{j-1} \not\in V({\cal T}^1)$, and
	$k = p$ or $w_{k+1} \not\in V({\cal T}^1)$;

	\item or
	$\{w_j, \ldots, w_k\} \cap V({\cal T}^1) = \varnothing$,
	$j = 1$ or $w_{j-1} \in V({\cal T}^1)$, and
	$k = p$ or $w_{k+1} \in V({\cal T}^1)$.
\end{itemize}

On the one hand, ${\cal G}^i$ is a subtree of ${\cal T}^1$ if and only if $i$ is even.
By Lemma~$\ref{lem:twotrees}$, $a({\cal G}^i) \ge a({\cal G}^{i+1})$ for every $i \in [k-1]$.
Therefore, $a({\cal H}_{h,\overline{t^1}} - {\cal T}^1) > a({\cal T}^1)$, which contradicts the fact that the maximum leading subtree of ${\cal F}$ is ${\cal T}^1$.

On the other hand, ${\cal G}^i$ is a subtree of ${\cal T}^1$ if and only if $i$ is odd.
Using Lemma~$\ref{lem:twotrees}$ again, we have that $a({\cal H}_{h,\overline{t^1}} - {\cal T}^1) > a({{\cal T}^1_{g^3,\overline{t^1}}})$.
However, the fact that the maximum leading subtree of ${\cal F}$ is ${\cal T}^1$ implies that $a({\cal G}^1) < a({{\cal T}^1_{g^3,\overline{t^1}}})$, which means that $a({\cal H}_{h,\overline{t^1}} - {\cal T}^1) > a({{\cal T}^1})$, a contradiction with the fact that the maximum leading subtree of ${\cal F}$ is ${\cal T}^1$.
\end{proof}

Let $(G,F,v) \in$ {\sc 1TP}.
Note that if we add to $G$ the necessary edges to make $v$ a universal vertex, then the obtained graph $H$ is trivially perfect. 
We call this completion the {\em universal TP completion of $(G,v)$}, or simply $U(G,v)$. 
Let ${\cal R}$ be a subtree of a tree of ${\cal F}$. We denote by
$\top{\cal R}{}$ the subtree of ${\cal R}$ induced by $r$ and the vertices of the subtrees ${\cal R}_u$ such that $u$ is a child of $r$ with $V({\cal R}_u) \cap N_G(v) = \varnothing$.

\subsection{Slim TP completion}

Given a trivially perfect graph $F$ and $S \subseteq V(F)$, we say that a graph $H$ is a {\em slim TP completion of $(F,S)$} if $H$ is a connected trivially perfect completion of $(F,S)$, $S$ is a clique of $H$ and $\ell_{\cal H}(s) \le |S|$ for every $s \in S$.
A slim TP completion is minimum if there is no other with less edges than it.
The task of computing a minimum slim TP completion of $(F,S)$ takes part of the algorithm for solving {\sc Trivially perfect $1$-completion} in polynomial time (Algorithm~\ref{alg:MinTPC}).
In such algorithm, the base of ${\cal H}$ will be the vertex $s \in S$ such that $\ell_{\cal H}(s) = |S|$.
We show in the sequel how to find a minimum slim TP completion in polynomial time.

\begin{lemma} \label{lem:minslim}
Let $F$ be a trivially perfect graph, let $S \subseteq V(F)$ be such that $V(F') \cap S \ne \varnothing$ for every connected component $F'$ of $F$ and let $u \in V(F)$ such that every proper ancestral of $u$ in ${\cal F}$ belongs to $S$. If $H$ is a minimum slim TP completion of $(F,S)$, then $H[V({\cal F}_u)]$ is a minimum slim TP completion of $(F_u,S)$.
\end{lemma}

\begin{proof}
Suppose by contradiction that $H^1 = H[V({\cal F}_u)]$ is not a minimum slim TP completion of $(F_u,S)$. We can assume that $u$ has been chosen maximizing its level in ${\cal F}$. 

First, consider that $u \not\in S$. If $S \cap V(F_u) = \varnothing$, then observe that
$V(F_u) \subset V({\cal H}_{u',u'})$, where $u' = \pi_{\cal F}(u)$, which means that no edges have been added joining vertices of $F_u$. Then, we can assume that $S_u := S \cap V(F_u) \ne \varnothing$.
By the definition of slim TP completion, every vertex of $S_u$ is ancestral of $u$ in ${\cal H}$, which means that $H^1$ is a slim TP completion of $(F_u,S)$. Recall that $V({\cal H}^1_{s,s}) = \{s\}$ if $s \in S_u$ is not the base of ${\cal H}^1$. Furthermore, the only vertices of $V(F_u)$ that are ancestrals of $\overline{h}$ in ${\cal H}$ are the vertices of $S_u$. Let $H^2$ be a minimum slim TP completion of $(F_u,S)$. We know that $|H^2| < |H^1|$.
Now, observe that if we replace $H^1$ by $H^2$ in $H$, the resulting graph is a slim TP completion of $(F,S)$ with less fill edges than $H$, which is a contradiction.

Consider now that $u \in S$. Denote by $u_1, \ldots, u_k$ the children of $u$ in ${\cal F}$ such that $F_{u_i}$ has vertices of $S$ for $i \in [k]$. Then, by the choice of $u$, it holds that $F_i = H[V({\cal F}_{u_i})]$ is a minimum slim TP completion of $({\cal F}_{u_i},S)$ for every $i \in [k]$.
Write ${\cal F}' = \{{\cal F}_1, \ldots, {\cal F}_k \}$.
Since $H^1$ is a slim TP completion of $(F_u,S)$, we conclude that ${\cal H}^2 = {\cal H}^1 - {\cal H}^1_{h^1,h^1}$ is a merge of ${\cal F}'$. Since $H^1$ is not minimum, by Lemma~\ref{lem:average}, ${\cal H}^2$ is not a leading merge of ${\cal F}'$.
Lemma~\ref{lem:average} also implies that there are vertices $u_1,u_2,u_3 \in V(H^2)$ such that $u_1 \in [h^2,u_2)_{{\cal H}^2}$, $u_2 \in (u_1,u_3)_{{\cal H}^2}$, $u_3 \in (u_2,\overline{h^2}]_{{\cal H}^2}$ and $|H'| < |H^2|$ where
${\cal H}' = \langle {\cal H}^2_{h^2,\pi(u_1)}, {\cal H}^2_{u_2,u_3}, {\cal H}^2_{u_1,\pi(u_2)}, {\cal H}^2 - {\cal H}^2_{h^2,u_3} \rangle$.
On the one hand, ${\cal H}^2_{u_1,u_3} = {\cal H}_{u_1,u_3}$. Then, Lemma~\ref{lem:average} implies that replacing ${\cal H}^2$ by ${\cal H}'$ in ${\cal H}$ yields a slim TP completion of $(F,S)$ with less edges than $H$.

On the other hand, ${\cal H}^2_{u_1,u_3} \ne {\cal H}_{u_1,u_3}$. Then, write $[u_1,u_3]_{\cal H} = \{w_1, \ldots, w_p\}$ and denote by ${\cal G}^1, \ldots, {\cal G}^q$ the partition of ${\cal H}_{u_1,u_3}$ into subtrees such that for every $i \in [q]$,
${\cal G}^i = {\cal H}_{w_j,w_k}$ is such that

\begin{itemize}
	\item either
	$\{w_j, \ldots, w_k\} \subseteq V({\cal H}^2_{u_1,u_3})$, 
	$j = 1$ or $w_{j-1} \not\in V({\cal H}^2_{u_1,u_3})$, and
	$k = p$ or $w_{k+1} \not\in V({\cal H}^2_{u_1,u_3})$;
	
	\item or
	$\{w_j, \ldots, w_k\} \cap V({\cal H}^2_{u_1,u_3}) = \varnothing$, 
	$j = 1$ or $w_{j-1} \in V({\cal H}^2_{u_1,u_3})$, and
	$k = p$ or $w_{k+1} \in V({\cal H}^2_{u_1,u_3})$.
\end{itemize}

Since $a({\cal H}^2_{u_2,u_3}) > a({\cal H}^2_{u_1,\pi(u_2)})$, we conclude that there is $i \in [k-1]$ such that $a({\cal G}^i) > a({\cal G}^{i+1})$. Then, apply Lemma~\ref{lem:twotrees} in these two trees, a slim TP completion of $(F,S)$ with less edges than $H$ is constructed, which is a contradiction.
\end{proof}

\begin{lemma} \label{lem:mincompaction}
The following hold for a connected trivially perfect graph $F$ and $S \subseteq V(F)$.

\begin{enumerate}[$(i)$]

\item If $f \not\in S$, then $H$ is a minimum slim TP completion of $(F,S)$ if and only if $H$ is the graph obtained from $F$ by adding the necessary edges to make every vertex of $S$ a universal vertex. \label{ite:fnotinS}

\item If $f \in S$, then $\langle {\cal F}_{f,f}, {\cal H}' \rangle$, where $H'$ is a minimum slim TP completion of $(F - V(\top{F}{}), S \setminus \{f\})$, is a minimum slim TP completion of $(F,S)$. Furthermore, ${\cal H}'$ is a leading merge of a minimum slim TP completions of the trees of ${\cal F} - \top{F}{}$. \label{ite:finS}

\end{enumerate}
\end{lemma}

\begin{proof}
\bigskip\noindent $(\ref{ite:fnotinS})$ Let $H'$ be a slim TP completion of $(F,S)$. Since $f \not\in S$, every vertex $s \in S$ is ancestral of $f$ in ${\cal H}'$. Note that $s \in S$ is ancestral of $f$ in ${\cal H}'$ if and only if $s$ is a universal vertex of $H'$.
Now, the result follows from the fact that the addition of the necessary edges to make every vertex of $S$ a universal vertex produces a slim TP completion of $(F,S)$.

\bigskip\noindent $(\ref{ite:finS})$ Since $f$ is a universal vertex of $F$, $f$ is a universal vertex of every TP completion of $(F,S)$, i.e., $f$ can be chosen as the root of any minimum slim TP completion of $(F,S)$. Let ${\cal H}$ be a minimum slim TP completion of $(F,S)$ with root $f$. Then, we can write ${\cal H} = \langle {\cal F}_{f,f}, {\cal H}' \rangle$. Note that $H'$ is a slim TP completion of $(F - V({\cal F}_{f,f}), S \setminus \{f\})$. Now, the minimality of ${\cal H}$ implies that ${\cal H'}$ is also minimum. Furthermore, since $f \in S$, Lemma~\ref{lem:minslim} implies every $H[F_u]$ is a minimum slim TP completion of $(F_u,S)$ for every child $u$ of $f$ in ${\cal F}$. Then, Lemma~\ref{lem:average} implies that ${\cal H}'$ is a leading merging of minimum slim TP completions of the trees of ${\cal F} - \top{F}{}$.
\end{proof}

\begin{lemma} \label{lem:slim}
Let $F$ be a trivially perfect graph of order $n$ and let $S \subseteq V(F)$ be such that $V(F') \cap S \ne \varnothing$ for every connected component $F'$ of $F$. A minimum slim TP completion of $(F,S)$ can be computed in $O(n^2)$ steps.
\end{lemma}

\begin{proof}
By Lemmas~\ref{lem:average} and~\ref{lem:minslim}, a minimum slim TP completion of $(F,S)$ can be constructed by 
by computing minimum slim TP completion of $(F',S \cap V(F'))$ for every connected component of $F$ and then doing a leading merge of the obtained trees.

According to Lemma~\ref{lem:mincompaction}, for every connected component $F'$ of $F$,
if $f' \not\in S$, then a minimum slim TP completion of $(F',S \cap V(F'))$ can be obtained by adding the necessary edges to make all vertices of $V(F') \cap S$ adjacent to all vertices of $F'$.
If $f' \in S$, then, applying recursion, a minimum slim TP completion of $(F',S \cap V(F'))$ is $\langle \top{F'}{}, {\cal H'} \rangle$, where $H'$ is a minimum slim TP completion of $(F' - V(\top{F'}{}), S \setminus \{f'\})$.

Since a leading merge can be computed in $O(n)$,
to add edges to make that some vertices become universal can also be done in $O(n)$ steps (considering the tree representation of a trivially perfect graph), and
the sum of the inputs of the recursive calls is smaller than $n$,
we conclude that the total time complexity $T(n)$ of this algorithm can be expressed as $T(n) = T(n-1) + O(n)$, which gives $T(n) = O(n^2)$.
\end{proof}

\subsection{Properties of a minimum TP completion}

Now, we present some properties of minimum TP completions that are used in the polynomial-time algorithm for solving the {\sc Trivially perfect $1$-completion} problem.

\begin{lemma} \label{lem:basic}
Let $G$ be a connected graph and let $(G,F,v) \in$ {\sc 1TP}. Then, every minimum TP completion $H$ of $G$ different of $U(G,v)$ having $v$ as the base satisfies the following properties where $S = N_G(v)$ and $\omega = \omega(F)$.

\begin{enumerate}[$(i)$]
	\item $h \in S \cup \{f^{\underline{1}}, \ldots, f^{\underline{\omega}}\}$. \label{ite:rootTiS}
	
	\item For every $u \in [h,v)_{\cal H}$, it holds $V({\cal H}_{u,u}) \subseteq V(F^{\underline{i}})$ for some $i \in [\omega]$. \label{ite:homogeneous}

	\item If $h \in V({\cal F}^{\underline{2}})$, then $h$ has a child in ${\cal H}$ belonging to $\{v\} \cup V({\cal F}^{\underline{1}})$. \label{ite:rootT2}

	\item If $h \in V({\cal F}^{\underline{2}})$, then the number of fill edges that are not incident to $h$ is less than $|{\cal F}^{\underline{2}}|$. \label{ite:F2root}

	\item If $\ell_{\cal H}(v) \ge 2$ and $\omega \ge 2$, then $fill(H) \ge |{\cal F}^{\underline{2}}|$. \label{ite:boundF2}

	\item If $\ell_{\cal H}(v) \ge 3$, then $h \in V({\cal F}^{\underline{1}}) \cup V({\cal F}^{\underline{2}})$. \label{ite:rootT1T2}

	\item If $\ell_{\cal H}(v) \ge 3$, then $|{\cal F}^{\underline{1}}| > |{\cal F}| - |{\cal F}^{\underline{1}}| - |{\cal F}^{\underline{2}}|$. \label{ite:universal}

	\item If $\ell_{\cal H}(v) \ge 3$, then $|{\cal F}^{\underline{1}}| > \frac{|{\cal F}|}{3}$. \label{ite:oneThird}
		
	\item For $u \in [h,v]_{\cal H}$, ${\cal H}_u$ is a minimum TP completion of $G[V({\cal H}_u)]$. \label{ite:hereditary}

	\item If $|V({\cal H}_{h,h})| = 1$, then there is $h^* \in (h,v)_{\cal H}$ with $|V({\cal H}_{h^*,h^*})| \ge 2$ and $i \in [\omega]$ such that $[h,h^*] \subset V(F^{\underline{i}})$. \label{ite:top}

	\item If $h \not\in \{f^{\underline{1}}, \ldots, f^{\underline{\omega}}\}$, then $vf^{\underline{i}} \not\in E(H)$ and ${\cal H} = \langle {\cal P}^i, {\cal Y} \rangle$ where $i \in [\omega]$ is such that $h \in V({\cal F}^i)$ and ${\cal P}^i$ is a minimum slim TP completion of $(F^{\underline{i}}, S \cap V(F^{\underline{i}}))$. \label{ite:slim}

\end{enumerate}
\end{lemma}

\begin{proof}
\bigskip\noindent $(\ref{ite:rootTiS})$ Suppose the contrary. Therefore, $h$ is a non-universal vertex of $V(F^{\underline{i}})$ for some $i \in [\omega]$. Observe that $h$ and $f^{\underline{i}}$ are not twins in $F^{\underline{i}}$ neither in $H$, but they are twins in $H[V(F^{\underline{i}})]$.

First, consider that $f^{\underline{i}}v \in E(H)$.
On the one hand, $f^{\underline{i}}$ is ancestral of $v$ in ${\cal H}$. Since $f^{\underline{i}}$ is not a universal vertex of $H$, $\omega \ge 2$ and there is $w \in V(F^{\underline{j}})$ for $j \ne i$ that is ancestral of $f^{\underline{i}}$ in ${\cal H}$ such that $|V({\cal H}_{w,w})| \ge 2$. We can assume that $w$ has been chosen with minimum level. Note that every vertex of $F^{\underline{k}}$ for $k \ne i$ has a fill edge to $h$ and every vertex of $V(F^{\underline{\ell}})$ for $\ell \ne j$ has a fill edge to $w$, which implies that $fill(H) \ge |V(F)| > fill(U(G,v))$, which is not possible.
On the other hand, $v$ is ancestral of $f^{\underline{i}}$ in ${\cal H}$. 
In this case, every vertex of $V(F^{\underline{i}}) \setminus S$ has a fill edge to $v$. Since $v \ne h$, $\omega \ge 2$ and then every vertex of $F^{\underline{j}}$ for $j \ne i$ has a fill edge to $h$. Therefore, we have that $fill(H) \ge fill(U(G,v))$.

Now, consider that $f^{\underline{i}}v \not\in E(H)$. Let $u$ be the vertex of $F^{\underline{i}}$ that is ancestral of $v$ in ${\cal H}$ with maximum level. Such vertex there exists because $h \in V(F^{\underline{i}})$.
Observe that $f^{\underline{i}} \in V({\cal H}_{u,u})$ and that for every $u' \in V(F^{\underline{i}})$ that is a proper ancestral of $u$ in ${\cal H}$, it holds ${\cal H}_{u',u'} = \{u'\}$, i.e., ${\cal H}_{h,h} = \{h\}$.
Now, construct a tree ${\cal H}^*$ by deleting $h$ from ${\cal H}$ and reinserting $h$ as the parent of all children of $u$ except the one that is ancestral of $v$. Now, it remains to observe that $H^*$ is also a TP completion of $G$ having less fill edges than $H$ since $E(H^*) \subset E(H)$, which is a contradiction. 

\bigskip\noindent $(\ref{ite:homogeneous})$
Suppose by contradiction that there is $u \in [h,v)_{\cal H}$ such that $V({\cal H}_{u,u})$ contains vertices of $V(F^{\underline{i}})$ and of $V(F^{\underline{j}})$ for different $i,j \in [\omega]$. Without loss of generality, we can assume that $u \in V(F^{\underline{i}})$. Let $s \in S \cap V(F^{\underline{j}})$ and $w \in V(F^{\underline{j}}) \cap V({\cal H}_{u,u})$.
We can assume that $w$ has been chosen minimizing its level in ${\cal H}$.
Since $sv \in E(H)$, we have that $s \not\in V({\cal H}_{u,u})$ and that $wv \not\in E(H)$. Since both $w$ and $s$ are adjacent to $f^{\underline{j}}$ in $H$, we conclude that $f^{\underline{j}} \in [h,u)_{\cal H}$. Let $w'$ be the vertex of $[f^{\underline{j}},u)_{\cal H}$ with maximum level such that $w' \in V(F^{\underline{j}})$
and $W$ be the induced subgraph of $H$ containing $w$ and the descendants of $w$ in ${\cal H}$ that belong to $V(F^{\underline{j}})$.

Now, let ${\cal H}'$ be the tree obtained from ${\cal H}$ by deleting $V(W)$ and adding the rooted tree ${\cal W}$ as a subtree of $w'$. Since $|H'| < |H|$, we have a contradiction.

\bigskip\noindent $(\ref{ite:rootT2})$ Suppose by contradiction that every vertex of $\{v\} \cup V({\cal F}^{\underline{1}})$ has level at least 3 in $\cal H$. Therefore, $fill(H) \ge |{\cal F}^{\underline{i}}| + 2 \left( \underset{j \in [\omega] \setminus \{2,i\}}{\sum} |{\cal F}^{\underline{j}}| \right)$ for some $i \in [\omega] \setminus \{1,2\}$. Since $|{\cal F}^{\underline{1}}| \ge |{\cal F}^{\underline{2}}|$, it holds that $fill(H) \ge |{\cal F}^{\underline{1}}| + |{\cal F}^{\underline{2}}| + |{\cal F}^{\underline{i}}| + 2 \left( \underset{j \in [\omega] \setminus \{1,2,i\}}{\sum} |{\cal F}^{\underline{j}}| \right) \ge |V(F)| > fill(U(G,v))$, which is a contradiction.

\bigskip\noindent $(\ref{ite:F2root})$ For every vertex $w \in V({\cal F} - {\cal F}^{\underline{2}})$, the fill edge $hw$ exists in $H$, which means that there are at least $|{\cal F} - {\cal F}^{\underline{2}}|$ fill edges in ${\cal H}$ with one extreme in $h$. Then, if the number of fill edges that are not incident to $h$ is at least $|{\cal F}^{\underline{2}}|$, it holds that $fill(H) \ge |V(G)| > fill(U(G,v))$, which is a contradiction.

\bigskip\noindent $(\ref{ite:boundF2})$ If $h \in V(F^{\underline{1}})$, then every vertex of $F^{\underline{2}}$ has a fill edge to $h$, and if $h \in V(F^{\underline{i}})$ for $i \ge 2$, then every vertex of $F^{\underline{1}}$ has a fill edge to $h$. Since $|{\cal F}^{\underline{1}}| \ge |{\cal F}^{\underline{2}}|$, the result does follow.

\bigskip\noindent $(\ref{ite:rootT1T2})$ Suppose by contradiction that $\ell_{\cal H}(v) \ge 3$ and $h \in V(F^{\underline{i}})$ for some $i \ge 3$. Denote by $v_2$ the ancestral of $v$ that is child of $h$ in ${\cal H}$.
If $v_2 \in V(F^{\underline{i}})$, then $wh$ and $wv_2$ are fill edges in $H$ for every $w \in V(F^{\underline{j}})$ for $j \ne i$, which means that $fill(H) \ge 2 \left( \underset{j \in [\omega] \setminus \{i\}}{\sum} |{\cal F}^{\underline{j}}| \right) \ge |V(G)| > fill(U(G,v))$ because $|{\cal F}^{\underline{i}}| \le |{\cal F}^{\underline{1}}|$, which is not possible. Therefore, $v_2 \in V(F^{\underline{k}})$ for some $k \in [\omega] \setminus \{i\}$.

Now, we know that $v_2$ is ancestral of every vertex of $V(F^{\underline{j}})$ for any $j \not\in \{i,k\}$. Therefore, $wh$ and $wv_2$ are fill edges in $H$ for every $w \in V(F^{\underline{j}})$ with $j \not\in \{i,k\}$. Furthermore, every vertex $w \in V(F^{\underline{k}})$ has a fill edge to $h$, which implies that $fill(H) \ge |{\cal F}^{\underline{k}}| + 2 \left(\underset{j \in [\omega] \setminus \{i,k\}}{\sum} |{\cal F}^{\underline{j}}| \right)  \ge |V(G)|  > fill(U(G,v))$ because $|{\cal F}^{\underline{1}}| \ge |{\cal F}^{\underline{i}}|$, which is a contradiction.

\bigskip\noindent $(\ref{ite:universal})$ Suppose by contradiction that $|{\cal F}^{\underline{1}}| \le |{\cal F}| - |{\cal F}^{\underline{1}}| - |{\cal F}^{\underline{2}}|$. If there is $j \in [\omega]$ such that $V({\cal H}_{h,v_2}) \subseteq V({\cal F}^{\underline{j}})$, where $v_2$ is the ancestral of $v$ that is child of $h$ in ${\cal H}$, then for every vertex $u$ of ${\cal F} - {\cal F}^{\underline{j}}$, the fill edges $uh$ and $uv_2$ exist in $H$, which means that

\[fill(H) \ge 2 \left(\underset{i \in [\omega] \setminus \{j\}}{\sum} |{\cal F}^{\underline{i}}| \right) \ge 2 \left(\underset{i \in [\omega] \setminus \{1\}}{\sum} |{\cal F}^{\underline{i}}| \right).\]

\noindent Since $\underset{i \in [\omega] \setminus \{1,2\}}{\sum} |{\cal F}^{\underline{i}}| \geq |{\cal F}^{\underline{1}}|$, $fill(H) \ge |V(G)| > fill(U(G,v))$, which is not possible. Then, without loss of generality, we can say that
$V({\cal H}_{h,h}) \subseteq V({\cal F}^{\underline{i}})$ and
$V({\cal H}_{v_2,v_2}) \subseteq V({\cal F}^{\underline{j}})$ for different $i,j \in [\omega]$. Therefore,
	
\[fill(H) \ge \min \{|{\cal F}^{\underline{i}}|, |{\cal F}^{\underline{j}}|\} + 2 \left( \underset{k \in [\omega] \setminus \{i,j\}}{\sum} |{\cal F}^{\underline{k}}| \right) \ge\]

\[|{\cal F}^{\underline{1}}| + \min \{|{\cal F}^{\underline{i}}|, |{\cal F}^{\underline{j}}|\} + \left( \underset{k \in [\omega] \setminus \{i,j\}}{\sum} |{\cal F}^{\underline{k}}| \right),\]

\noindent because $\underset{k \in [\omega] \setminus \{i,j\}}{\sum} |{\cal F}^{\underline{i}}| \geq |{\cal F}^{\underline{1}}|$. Since $|{\cal F}^{\underline{1}}| + \min \{|{\cal F}^{\underline{i}}|, |{\cal F}^{\underline{j}}|\} \ge |{\cal F}^{\underline{i}}| + |{\cal F}^{\underline{j}}|$, it holds that $fill(H) \ge |V(G)| > fill(U(G,v))$, which is a contradiction.

\bigskip\noindent $(\ref{ite:oneThird})$ By~$(\ref{ite:universal})$, we know that $|{\cal F}^{\underline{1}}| > |{\cal F}^{\underline{3}}| + \ldots + |{\cal F}^{\underline{\omega}}|$. Since $|{\cal F}^{\underline{1}}| \ge |{\cal F}^{\underline{2}}|$, it holds that $|{\cal F}^{\underline{1}}| > \frac{|{\cal F}|}{3}$.

\bigskip\noindent $(\ref{ite:hereditary})$ Suppose that there is a minimum TP completion $R$ of $G[V({\cal H}_u)]$ having less fill edges than ${\cal H}_u$. Now, observe that $\langle {\cal H}_{h,\pi_{\cal H}(u)}, {\cal R} \rangle$ is a TP completion of $G$ with less fill edges than $H$, which is a contradiction.

\bigskip\noindent $(\ref{ite:top})$ Let $i \in [\omega]$ such that $h \in V(F^{\underline{i}})$ and let $x \in [h,v]_{\cal H}$ such that $x \not\in V(F^{\underline{i}})$ but every proper ancestral of $x$ in ${\cal H}$ belongs to $V(F^{\underline{i}})$. Suppose by contradiction that $|V({\cal H}_{h^*,h^*})| = 1$ for every proper ancestral of $x$ in ${\cal H}$. 
Note that there is a fill edge joining every vertex not in $F^{\underline{i}}$ to $h$; and that
there is a fill edge joining every vertex in $F^{\underline{i}}$ to $x$.
Therefore, $fill(H) \ge |{\cal F}| > fill(U(G,v))$, which is a contradiction.

\bigskip\noindent $(\ref{ite:slim})$
By~$(\ref{ite:rootTiS})$, we know that $h \in S$. Let $i \in [\omega]$ such that $h \in V({\cal F}^i)$.

If $|S \cap V(F^{\underline{i}})| = 1$, then we claim that $V(\top{H}{}) = V(F^{\underline{i}})$, which proof the result for this case.
Suppose the contrary and let $w$ be a vertex of $V(F^{\underline{i}}) \setminus V(\top{H}{})$ with minimum level in ${\cal H}$. Denote by $W$ the subgraph of $F^i$ induced by the descendant vertices of $w$ in ${\cal H}$.
Now, let ${\cal H}'$ be the tree obtained from ${\cal H}$ by deleting $V(W)$ and adding the rooted tree ${\cal W}$ as a subtree of $h$. Since $|H'| < |H|$, we have a contradiction.

Then, assume that $|S \cap V(F^{\underline{i}})| \ge 2$. Therefore, $|V(\top{H}{})| = 1$.
By~$(\ref{ite:top})$, there is $h^* \in (h,v)_{\cal H}$ with $|V({\cal H}_{h^*,h^*})| \ge 2$ such that $[h,h^*] \subset V(F^{\underline{i}})$. Choosing $h^*$ with minimum level in ${\cal H}$, we conclude that $V({\cal H}_{h',h'}) = \{h'\}$ for every $h' \in [h,h^*)_{\cal H}$. Furthermore, $f^{\underline{i}} \not\in [h,h^*]_{\cal H}$, because otherwise, $f^{\underline{i}}$ could be chosen as the root of ${\cal H}$.
Since $f^{\underline{i}}$ is a universal vertex of $F^i$, $|V({\cal H}_{h^*,h^*})| \ge 2$ implies that $f^{\underline{i}} \in V({\cal H}_{h^*,h^*}) \setminus \{h^*\}$, which means that $f^{\underline{i}}v \not\in E(H)$.
Therefore, every vertex of $S \cap V(F^{\underline{i}})$ is ancestral of $v$ in ${\cal H}$ and $\ell_{\cal H}(s) \le |S \cap V(F^{\underline{i}})|$ for every $s \in S \cap V(F^{\underline{i}})$, which means that ${\cal H}_{h,h^*}$ is a slim TP completion of $(F^{\underline{i}},S \cap V(F^{\underline{i}}))$ and ${\cal H} = \langle {\cal H}_{h,h^*}, {\cal Y} \rangle$. Finally, the fact that $H$ is a minimum TP completion of $G$ implies that ${\cal H}_{h,h^*}$ is a minimum slim TP completion of $(F^{\underline{i}},S \cap V(F^{\underline{i}}))$.
\end{proof}

\subsection{Computing a minimum TP completion}

We begin by presenting the general idea of the proposed algorithm.
Let $(G,F,v) \in$ {\sc 1TP} and let $H$ be a minimum TP completion of $G$.
Because of items~$(\ref{ite:rootTiS})$,~$(\ref{ite:rootT1T2})$ and~$(\ref{ite:slim})$ of Lemma~\ref{lem:basic}, either $\ell_{\cal H}(v) \le 2$ or
there are at most $4$ possibilities to $h$, namely, $h \in \{f^{\underline{1}},f^{\underline{2}},s_1,s_2\}$, where $s_i$ is the root of a minimum slim TP completion $P^i$ of $F^{\underline{i}}$ for $i \in [2]$.
We will see that to construct all trees ${\cal T}$ such that $T$ is a TP completion of $G$ and $\ell_{\cal T}(v) \le 2$ can be done in polynomial time.
To the other cases, we can write ${\cal H} = \langle \top{H}{}, {\cal H}^* \rangle$, where either $\top{H}{} = \top{F}{\underline{i}}$ 
for $i \in [2]$ or $\top{H}{} = \top{P}{i}$ for $i \in [2]$ and ${\cal H}^*$ is
a minimum TP completion of $G - \top{H}{}$.
These ideas lead to an algorithm for computing a minimum TP completion of $G$.
However, because of the recursion, we cannot guarantee a polynomial-time complexity. Therefore, in Algorithm~\ref{alg:MinTPC}, we exploit these and other properties in order to replace this big recursive call by few recursive calls with small inputs in order to obtain a polynomial-time number of steps.

We need some definitions.
Let $G$ be a connected graph and $v$ a vertex of $V(G)$ such that $(G,F,v) \in$ {\sc 1TP}, and let $k = \min \{2, \omega(F)\}$.
For $i \in [\omega(F)]$, $\Upsilon_i(G,v)$ is a family containing at most two members.
One of the members of $\Upsilon_i$ is $\top{F}{\underline{i}}$, and if $|N_G(v) \cap V(F^{\underline{i}})| = 1$, then the second member exists and it is the minimum slim TP completion of $(F^{\underline{i}}, N_G(v) \cap V(F^{\underline{i}}))$.
Denoting by $J^1$ a minimum slim TP completion of $(F^{\underline{i}},N_G(v) \cap V(F^{\underline{i}}))$, define $\Lambda(G,v) = \{ U(G,v), {\cal J}^1, \Upsilon_1, \Upsilon_2 \}$.
Let $H$ be a TP completion of $G$.
We will write ${\cal H} = \langle {\cal X}_1, \ldots, {\cal X}_k \rangle_\Lambda$ if for every $i \in [k]$, it holds that ${\cal X}_i \in \Lambda(G-({\cal X}_1 \cup \ldots {\cal X}_{i-1}),v)$.

Using the above definitions, we can state a result that synthesizes and is stronger than the ideas discussed in the first paragraph of this section.

\begin{corollary} \label{cor:minTPC}
Let $(G,F,v) \in$ {\sc 1TP}. If $H$ is a minimum TP completion of $G$, then there is ${\cal X}^1 \in \Lambda(G,v)$ such that ${\cal H} = \langle {\cal X}^1,{\cal Y} \rangle$ where ${\cal Y}$ is a minimum TP completion of $G - X^1$. Furthermore, we can write ${\cal H} = \langle {\cal X}^1, \ldots, {\cal X}^k \rangle_\Lambda$.
\end{corollary}

\begin{proof}
It is a consequence of items~$(\ref{ite:rootTiS})$,~$(\ref{ite:homogeneous})$,~$(\ref{ite:rootT1T2})$ and~$(\ref{ite:slim})$ of Lemma~\ref{lem:basic}.
\end{proof}

We discuss now some subroutines that are used in Algorithm~\ref{alg:MinTPC}. We call attention to the fact that the variables are passed to these subroutines by reference.

\begin{itemize}
\item {\sc MinOf}(${\cal Q}$): ${\cal Q}$ is a family of graphs all having the same vertex set. The algorithm returns the graph of ${\cal Q}$ with minimum number of edges.

\item {\sc Add}(${\cal T},\Gamma$): ${\cal T}$ is a rooted tree and $\Gamma$ is a family of rooted trees. 
The algorithm adds ${\cal T}$ to $\Gamma$.

\item {\sc FindTop}$(G, v, {\cal R}, \Upsilon, {\cal Q})$: Algorithm~\ref{alg:findTop}.

\item {\sc SlimTPC}($F, S$): $F$ is a trivially perfect graph and $S \subseteq V(F)$. The algorithm returns a minimum slim TP completion of $(F,S)$.

\end{itemize}

\begin{algorithm}[h]

	\caption{{\sc FindTop}}
	\label{alg:findTop}
	
	\SetKwInOut{Input}{input}
	\SetKwInOut{Output}{output}

	\Input{$(G,F,v) \in$ {\sc 1TP} and a subtree ${\cal R}$ of ${\cal F}^{\underline{1}}$ (it also receives references to the family $\Upsilon$ and the set of TP completions ${\cal Q}$ of $G$)}
	
	\Output{It updates ${\cal Q}$ and $\Upsilon$}

	${\cal T} := \langle {\cal R} , U(G - R,v) \rangle$ \label{lin:hv}
	
	$k = \min \{2, \omega(G - v)\}$

	\For{$i \in [k]$}{ \label{lin:for}
		$\Upsilon_i := \Upsilon_i(G - R, v)$ \label{lin:upsilon}

		\For{${\cal J} \in \Upsilon_i$}{
			${\cal T}' := \langle {\cal R}, {\cal J}, U(G - (R \cup J)) \rangle$

			\If{$|T'| < |T|$}{
				$T := T'$
			}
		}
	}
	{\sc Add}$(T, {\cal Q})$ \label{lin:endtopfind}
\end{algorithm}

\begin{lemma} \label{lem:findtop}
Given $(G,F,v) \in$ {\sc 1TP} and a subtree ${\cal R}$ of ${\cal F}^{\underline{1}}$, Algorithm~$\ref{alg:findTop}$ computes $\Upsilon_1(G - R, v), \Upsilon_2(G - R, v)$ and a minimum TP completion of $G$ of the form $\langle {\cal R}, {\cal T} \rangle$ with $\ell_{\cal T}(v) \le 2$ in $O(n)$ steps where $n$ is the order of $G$.
\end{lemma}

\begin{proof}
We can assume that we are using the tree representations of the graphs.	
Observe that we can rewrite the statement of the lemma as follows: given $(Y,Z,v) \in$ {\sc 1TP}, Algorithm~$\ref{alg:findTop}$ computes $\Upsilon_1(Y,v), \Upsilon_2(Y,v)$ and a minimum TP completion $T$ of $Y$ with $\ell_{\cal T}(v) \le 2$ in $O(n^2)$ steps where $n$ is the order of $Y$. 

Note that if $\{u\} = N_Y(v) \cap V(Z^{\underline{i}})$ for some $i \in [2]$, then the minimum slim TP completion of $(Z^{\underline{i}}, N_Y(v) \cap V(Z^{\underline{i}}))$ is the graph obtained by making $u$ adjacent to all vertices of 
$Z^{\underline{i}}$. This can be clearly done in $O(n)$ steps. Therefore, $\Upsilon_i(G - R, v)$ can be computed in $O(n)$ steps in line~\ref{lin:upsilon}.

It is clear that a minimum TP completion $T$ of $Y$ such that $\ell_{\cal T}(v) = 1$ is $U(Y,v)$. Such tree is computed in $O(n)$ steps in line~\ref{lin:hv}. Now, by Corollary~\ref{cor:minTPC}, a minimum TP completion $T$ of $Y$ such that $\ell_{\cal T}(v) = 2$ is the minimum tree $\langle {\cal X}^1, U(Y - X^1,v) \rangle$ which can be formed choosing ${\cal X}^1$ among the members of $\Upsilon_1(Y, v) \cup \Upsilon_2(Y, v)$. This is done in lines~\ref{lin:for} to~\ref{lin:endtopfind} of this algorithm spending $O(n)$ steps since we are working with their tree representations.
\end{proof}

Despite that the {\sc Trivially Perfect $1$-completion} has as input a graph $G$, the input of Algorithm~\ref{alg:MinTPC} is formed by a graph $G$ and a vertex $v \in V(G)$ to guarantee that the same vertex $v$ be used in all recursive calls.

\newcounter{numLinhas}

\todo{não existem mais a família upsilon}

\begin{algorithm}
	
	\caption{{\sc MinTPC}}
	\label{alg:MinTPC}
	
	\SetKwInOut{Input}{input}
	\SetKwInOut{Output}{output}
	\SetKwRepeat{Do}{do}{while}
	
	\Input{Graph $G$ and $v \in V(G)$ such that $G - v$ is a trivially perfect graph}
	\Output{A minimum TP completion of $G$}

	$F := G - v$

	$S := N_G(v)$

	$\Upsilon := \Gamma := \varnothing$

	{\sc FindTop}$(G, v, \varnothing, \Upsilon, {\cal Q})$ \label{lin:universal}

	${\cal B} := {\cal F}^{\underline{1}} - \top{F}{\underline{1}}$ \label{lin:B1}

	\If{$|{\cal F}^{\underline{1}}| < \frac{2|{\cal F}|}{3}$}{ \label{lin:23begin}
			
		\For{${\cal D}^2 \in \Upsilon_2$}{

			{\sc Add(}$\langle \top{F}{\underline{1}}, \top{B}{\underline{1}}, {\cal D}^2 \rangle, \Gamma)$ \label{lin:G1G1G2}

			\If{$|S \cap V(B^{\underline{1}})| = 1$}{
				{\sc Add(}$\langle \top{F}{\underline{1}},$ {\sc slimTPC}$({\cal B}^{\underline{1}},S \cap V({\cal B}^{\underline{1}})), {\cal D}^2 \rangle, \Gamma)$ \label{lin:G1C1G2}
			}

			\For{${\cal D}^1 \in \Upsilon_1$}{ \label{lin:E}
				${\cal L} :=$ {\sc MinOf}$(\{\langle {\cal D}^2, {\cal D}^1 \rangle, \langle {\cal D}^1,  {\cal D}^2 \rangle\})$ \label{lin:G2G1}

				{\sc Add}$(\langle {\cal L} , U(G - L,v), {\cal Q} \rangle)$ \label{lin:v3}
	
				\If{$|{\cal D}^1| = |{\cal F}^{\underline{1}}|$}{
					{\sc Add(}${\cal L}, \Gamma)$ \label{lin:G1G2big}

				}\ElseIf{$|{\cal L}| \ge \frac{|{\cal F}|}{6}$}{ \label{lin:if56}
						{\sc Add(}$\langle {\cal L} \ , \ ${\sc MinTPC}$(G - L, v) \rangle, {\cal Q})$ \label{lin:G1G2small}
				}
			}
		}
		\If{$|{\cal F}^{\underline{1}} \cap S| = 1$}{

			{\sc Add(slimTPC}$({\cal F}^{\underline{1}},S), \Gamma)$ \label{lin:G1comp}

			{\sc Add(}$\langle \top{F}{\underline{1}}, \ ${\sc slimTPC}$({\cal B}^{\underline{1}},S) \rangle, \Gamma)$ \label{lin:G11comp1}
		}
		\ElseIf{$|{\cal F}^{\underline{1}} \cap S| = 2$}{
			{\sc Add(slimTPC}$({\cal F}^{\underline{1}},S),\Gamma)$ \label{lin:G11comp2}
		}
		\For{${\cal T} \in \Gamma$}{ 
			\If{$|{\cal T}| \ge \frac{|{\cal F}|}{3}$}{ \label{lin:if23}
				{\sc add(}$\langle {\cal T} \ , \ ${\sc MinTPC}$(G - T, v) \rangle, {\cal Q})$ \label{lin:rec23}
			}
		}
		{\sc FindTop}$( G, v, \top{F}{\underline{1}}, \Upsilon, {\cal Q} )$ \label{lin:v2F3}

		{\sc FindTop}$(G, v, \langle \top{F}{\underline{1}}, \top{B}{\underline{1}} \rangle, \Upsilon, {\cal Q} )$ \label{lin:v3F3}

		\Return {\sc MinOf}$({\cal Q})$ \label{lin:best23} \label{lin:23end}
	}
	\tcc{It continues}
		
	\setcounter{numLinhas}{\value{AlgoLine}}		
\end{algorithm}

\begin{algorithm}
	\addtocounter{algocf}{-1}	
	\setcounter{AlgoLine}{\value{numLinhas}}

	\caption{continuation of {\sc MinTPC}}

	${\cal R} := \varnothing$ \label{lin:begin23}

	${\cal B} := {\cal F}$
	
	\While{$|{\cal R}| < \frac{|{\cal F}|}{3}$ {\bf and} $|{\cal B}^{\underline{1}}| \ge \frac{|{\cal B}|}{3}$}  { \label{lin:while}

		{\sc FindTop}$(G, v, {\cal R}, \Upsilon, {\cal Q})$ \label{lin:v}

		\For{${\cal Y} \in {\Upsilon}_2$}{ \label{lin:Y2}

			\If{$|\langle {\cal R}, {\cal Y}\rangle| \ge \frac{|{\cal F}|}{3}$ {\bf and} $\left(\pi_{{\cal F}^{\underline{1}}}(b^{\underline{2}}) = \overline{r} \ {\bf or} \ |\langle {\cal R}_{r,\pi(\overline{r})} , {\cal Y}\rangle| < \frac{|{\cal F}|}{3}\right)$}{ \label{lin:F2minlevel}

				\sc Add($\langle {\cal R}, {\cal Y} \rangle, \Gamma^1)$ \label{lin:RF2}		
			}
		}
		\If{$b^{\underline{1}} \not\in S$}{ \label{lin:f1S}
			
			${\cal T} := $ {\sc slimTPC}$({\cal B}^{\underline{1}}, S \cap V({\cal B}^{\underline{1}}))$ \label{lin:compF1}

			\If{$\frac{|{\cal R}|}{3} + |{\cal T}| \ge \frac{|{\cal F}|}{3}$}{ \label{lin:save13}
			
				{\sc Add}$(\langle {\cal R}, {\cal T}\rangle,\Gamma^2)$ \label{lin:saveComp}
			}
		}

		${\cal R} := \langle {\cal R}, \top{B}{\underline{1}} \rangle$ \label{lin:F1}

		${\cal B} := {\cal B} - \top{B}{\underline{1}}$ \label{lin:redefineB}
	}

	\tcc{For ${\cal M}^{\ell} \in \Gamma^1$, denote by ${\cal B}^{\ell}$ the instance of ${\cal B}^{\underline{2}}$ at the moment that ${\cal M}^{\ell}$ is added to $\Gamma^1$.}

	\For{${\cal M}^i \in \Gamma^1$} { \label{lin:M}

		\If{there are ${\cal M}^j, {\cal M}^k \in \Gamma^1$ {\bf with} $|{\cal B}^i| < |{\cal B}^j| < |{\cal B}^k|$ }{ \label{lin:2parents}

			$r' :=$ vertex such that $\pi_{{\cal F}^{\underline{1}}}(r') = \overline{m^i}$

			\For{$r'' \in [r',\pi_{{\cal F}^{\underline{1}}}(b^{\underline{2}})]_{\cal R}$}{
				{\sc Add}$(\langle {\cal M}^i, {\cal R}_{r',r''}, U(G- (V({\cal M}^i) \cup V({\cal R}_{r',r''})),v) \rangle, {\cal Q})$ \label{lin:G1a}
			}
		}
		\Else{
			{\sc Add($\langle {\cal M}^i,$ MinTPC}$(G - M^i, v) \rangle, {\cal Q})$ \label{lin:G1b}
		}
	}

	$\Gamma^3 := $ set formed by the trees of $\Gamma^2$ with minimum number of vertices \label{lin:G3aa}

	$\Gamma^4 := $ set formed by the trees of $\Gamma^2 \setminus \Gamma^3$ with minimum number of vertices

	add to $\Gamma$ one of the trees $\cal Y$ of $\Gamma^3$ such that $|Y|$ is minimum \label{lin:G3a}

	add to $\Gamma$ one of the trees $\cal Y$ of $\Gamma^4$ such that $|Y|$ is minimum \label{lin:G3b}

	\tcc{It continues}

	\setcounter{numLinhas}{\value{AlgoLine}}		
\end{algorithm}

\begin{algorithm}
\addtocounter{algocf}{-1}	
\setcounter{AlgoLine}{\value{numLinhas}}

\caption{continuation of {\sc MinTPC}}

	\If{$|{\cal R}| \ge \frac{|{\cal F}|}{3}$}{ \label{lin:testR}
		{\sc Add(${\cal R} , \Gamma)$} \label{lin:saveR}
	}
	\Else{ \label{lin:Rsmall}
		
		\For{${\cal C} \in \Upsilon_2$}{ \label{lin:T2a}

		\For{$r^* \in [\pi_{{\cal F}^{\underline{1}}}(b^{\underline{2}}), \overline{r}]_{\cal R}$}{
			${\cal D} := \langle {\cal R}_{r,\pi(r^*)} , {\cal C} , {\cal R}_{r^*}\rangle$
			
			{\sc Add(}$\langle {\cal D}, U(G - {\cal D})\rangle, {\cal Q})$ \label{lin:QF2a}
		}
	}
}
\For{${\cal Z} \in \Gamma$}{ \label{lin:bestTop}
	
	{\sc Add($\langle {\cal Z} ,$ MinTPC}$(G - Z, v)\rangle, {\cal Q})$ \label{lin:recR}

	$b := $ root of the maximum tree of ${\cal F} - {\cal Z}$ \label{lin:lastB2}

	\For{${\cal C} \in \Upsilon_2$}{ \label{lin:T2}

		${\cal D} := U(G,v)$
	
		\For{$z^* \in [\pi_{{\cal F}}(b), \overline{z}]_{\cal Z}$}{ \label{lin:merge}
			
			\If{$|E(\langle {\cal Z}_{z,\pi(z^*)} , {\cal C} , {\cal Z}_{z^*}\rangle)| < |E({\cal D})|$}{
				
				${\cal D} := \langle {\cal Z}_{r,\pi(z^*)} , {\cal C} , {\cal Z}_{z^*}\rangle$
			}
		}
				
		{\sc Add(}$\langle {\cal D} ,$ {\sc MinTPC}$(G - D, v)\rangle, {\cal Q})$ \label{lin:QF2}
		}
	}

	\Return {\sc MinOf}$({\cal Q})$ \label{lin:best}

\setcounter{numLinhas}{\value{AlgoLine}}		
\end{algorithm}

\begin{theorem}
Algorithm~$\ref{alg:MinTPC}$ is correct.
\end{theorem}

\begin{proof}
Let $H$ be a minimum TP completion of $G$ and set $v$ as the base of ${\cal H}$. According to Corollary~\ref{cor:minTPC}, there is ${\cal X}^1 \in \Lambda(G,v)$ such that ${\cal H} = \langle {\cal X}^1,{\cal Y} \rangle$ where ${\cal Y}$ is a minimum TP completion of $G - X^1$.
The algorithm is divided into two cases. If $|{\cal F}^{\underline{1}}| < \frac{2|{\cal F}|}{3}$, then it finishes in line~\ref{lin:best23}; otherwise, it finishes in line~\ref{lin:best}.
In both lines, a call to {\sc MinOf} is made where the family ${\cal Q}$ containing some TP completions of $G$ is passed as parameter. Therefore, we have to show that $H$ belongs to ${\cal Q}$ when one of these two lines is reached.

By Lemma~\ref{lem:findtop}, the possibilities for $\ell_{\cal H}(v) \le 2$ are considered in line~\ref{lin:universal} because in the call to {\sc FindTop} done in this line we have $V({\cal R}) = \varnothing$. From now on, we assume that $\ell_{\cal H}(v) \ge 3$. Hence, Lemma~\ref{lem:basic}~$(\ref{ite:rootT1T2})$ implies that $h \in V({\cal F}^{\underline{1}}) \cup V({\cal F}^{\underline{2}})$, more precisely, ${\cal X}^1 \in \Upsilon_1 \cup \Upsilon_2$. We denote by $v_2$ the ancestral of $v$ at level 2. If $\ell_{\cal H}(v) \ge 4$, denote by $v_3$ the ancestral of $v$ at level~3.

\bigskip \noindent {\bf Case 1:} $|{\cal F}^{\underline{1}}| < \frac{2|{\cal F}|}{3}$ (lines~\ref{lin:23begin} to~\ref{lin:23end}).

In lines~\ref{lin:G1G1G2},~\ref{lin:G1C1G2},~\ref{lin:G1G2big},~\ref{lin:G1comp},~\ref{lin:G11comp1} and~\ref{lin:G11comp2}, some rooted trees are saved in the family $\Gamma$. In line~\ref{lin:rec23}, for each ${\cal T}$ added to $\Gamma$ with at least $\frac{|{\cal F}|}{3}$ vertices, a recursive call with input $G - {\cal T}$ is done, which returns as result a tree ${\cal T}^*$. Then, the TP completion $\langle {\cal T}, {\cal T}^* \rangle$ is saved in ${\cal Q}$ as a candidate to be a minimum TP completion of $G$. Other TP completions of $G$ are also saved in ${\cal Q}$ in lines~\ref{lin:v3} and~\ref{lin:G1G2small} and in the calls to {\sc FindTop} (lines~\ref{lin:universal},~\ref{lin:v2F3} and~\ref{lin:v3F3}). In line~\ref{lin:23end}, a TP completion contained in ${\cal Q}$ with minimum number of edges is chosen as a minimum TP completion of $G$. Therefore, we have to prove that ${\cal H}$ is added to ${\cal Q}$ in one of these lines.

We need to show that a tree ${\cal T}$ saved in $\Gamma$ cannot satisfies ${\cal H} = \langle {\cal T}, {\cal T}^* \rangle$ if $|{\cal T}| < \frac{|{\cal F}|}{3}$.
By Lemma~\ref{lem:basic}~$(\ref{ite:oneThird})$, each tree added in one of the lines~\ref{lin:G1G2big},~\ref{lin:G1comp},~\ref{lin:G11comp1} and~\ref{lin:G11comp2} has at least $\frac{|{\cal F}|}{3}$ vertices because it contains $V(F^1)$.
Note that each tree ${\cal T}$ added to $\Gamma$ in line~\ref{lin:G1G1G2} or~\ref{lin:G1C1G2} is formed from vertices of three trees ${\cal T}^1, {\cal T}^2$ and ${\cal T}^3$ such that ${\cal T}^1$ is a subtree of ${\cal F}^{\overline{1}}$ and ${\cal T}^2$ and ${\cal T}^3$ are subtrees of different trees of ${\cal F}$.
Thus, suppose by contradiction that ${\cal T} = {\cal H}_{h,v_3}$ was added to $\Gamma$ in line~\ref{lin:G1G1G2} or~\ref{lin:G1C1G2} having less than $\frac{|{\cal F}|}{3}$ vertices.
Since $h \in V(F^{\underline{1}})$, $H$ has at least $\frac{|{\cal F}|}{3}$ fill edges with one extreme in $h$ because $|{\cal F}^{\underline{1}}| < \frac{2|{\cal F}|}{3}$.
Furthermore, $H$ has other $\frac{2|{\cal F}|}{3}$ fill edges, each with one extreme in $v_2$ or $v_3$ and the other extreme in ${\cal H} - {\cal H}_{h,v_3}$, because every vertex of ${\cal H} - {\cal H}_{h,v_3}$ has a non-neighbor in $G$ belonging to $\{v_2,v_3\}$ since $v_2$ and $v_3$ belong to different tres of ${\cal F}$ and $|{\cal H}_{h,v_3}| < \frac{|{\cal F}|}{3}$, which is a contradiction because this implies that $fill(H) \ge |V(F)| > fill(U(G,v))$.

First, consider that $h \in V(F^{\underline{2}})$. Since $\ell_{\cal H}(v) \ge 3$, Lemma~\ref{lem:basic}~$(\ref{ite:rootT2})$ implies that $v_2 \in V(F^{\underline{1}})$. Because of Lemma~\ref{lem:basic}~$(\ref{ite:slim})$, the possible cases for $\top{H}{}$ are the members of $\Upsilon_2$. Recall that $\Upsilon_i$ for $i \in [\omega(F)]$ was constructed in line~\ref{lin:universal} and contains at most two members, one is $\top{F}{\underline{i}}$ and the other is the slim TP completion of $(F^{\underline{i}}, N_G(v) \cap V(F^{\underline{i}}))$ if $|S \cap V(F^{\underline{i}})| = 1$.

All possibilities for $V(\top{H}{}) \subseteq V(F^{\underline{2}})$ and $V({\cal H}_{v_2,v_2}) \subseteq V(F^{\underline{1}})$ are considered in lines~\ref{lin:v3}, \ref{lin:G1G2big} and~\ref{lin:G1G2small}. In fact, note that line~\ref{lin:G1G2small} is executed only if $|{\cal H}_{h,v_2}| \ge \frac{|{\cal F}|}{6}$, condition that is checked in line~\ref{lin:if56}.
Therefore, suppose by contradiction that $|{\cal H}_{h,v_2}| < \frac{|{\cal F}|}{6}$.
Since the case $\ell_{\cal H}(v) = 3$ is covered in line~\ref{lin:v3}, we can assume that $\ell_{\cal H}(v) \ge 4$.
Since $h \in V(F^{\overline{2}})$ and $v_2  \in V(F^{\overline{1}})$, $H$ has at least $\frac{5|{\cal F}|}{6}$ fill edges each of them with exactly one extreme in $\{h,v_2\}$ because each vertex of ${\cal H} - {\cal H}_{h,v_2}$ has a fill edge to $h$ or to $v_2$.
Now, Lemma~\ref{lem:basic}~$(\ref{ite:oneThird})$ implies that $|W^1| \ge \frac{|{\cal F}|}{6}$ where $W^1 = V(F^{\underline{1}}) \setminus V({\cal H}_{v_2,v_2})$.
The assumptions that $|{\cal F}^{\underline{1}}| < \frac{2|{\cal F}|}{3}$ and $|{\cal H}_{h,v_2}| < \frac{|{\cal F}|}{6}$ imply that $|W^2| \ge \frac{|{\cal F}|}{6}$ where
$W^2 = V({\cal H}_{v_3}) \setminus (V(F^{\underline{1}}) \cup V(F^{\underline{2}}))$.
Since $\ell_{\cal H}(v) \ge 4$, if $v_3 \in V(F^{\underline{1}})$, then $H$ has at least $\frac{|{\cal F}|}{6}$ fill edges all of them having one extreme in $v_3$ and the other in $W^2$;
and if $v_3 \not\in V(F^{\underline{1}})$, then $H$ has at least $\frac{|{\cal F}|}{6}$ fill edges all of them having one extreme in $v_3$ and the other in $W^1$. Therefore, $fill(H) \ge |V(G)| > fill(U(G,v))$, which is a contradiction.

Consider now that $h \in V(F^{\underline{1}})$.
By Corollary~\ref{cor:minTPC}, ${\cal X}^1 \in \{\top{F}{\underline{1}}, {\cal J}\}$, where ${\cal J}$ is a minimum slim TP completion of $(F^{\underline{1}}, N_G(v) \cap V(F^{\underline{1}}))$.
First, assume that ${\cal X}^1 = {\cal J}$.
The case $|S \cap V(F^{\underline{1}})| = 1$ is considered in line~\ref{lin:G1comp}.
The case $|S \cap V(F^{\underline{1}})| = 2$ is considered in lines~\ref{lin:G11comp1} and~\ref{lin:G11comp2}.
At last, assume by contradiction that $|S \cap V(F^{\underline{1}})| \ge 3$. Therefore, $\{h,v_2,v_3\} \subset V(F^1)$ and every vertex of $F - F^{\underline{1}}$ has three fill edges each one with an extreme in each vertex of $\{h,v_2,v_3\}$. Since $|{\cal F}^{\underline{1}}| < \frac{2|{\cal F}|}{3}$, it holds that $|{\cal F} - {\cal F}^{\underline{1}}| \ge \frac{|{\cal F}|}{3}$ and then that $fill(H) \ge |V(F)| > fill(U(G,v))$, which is a contradiction.

Now, assume that ${\cal X}^1 = \top{F}{\underline{1}}$. The possible cases for $v_2 \in V(F^{\underline{2}})$ are considered in lines~\ref{lin:v3}, \ref{lin:G1G2big} and~\ref{lin:G1G2small}.
As noted above, line~\ref{lin:G1G2small} is executed only if $|{\cal H}_{h,v_2}| \ge \frac{|{\cal F}|}{6}$, condition that is checked in line~\ref{lin:if56}. The same proof used above does hold here to show that $|{\cal H}_{h,v_2}| < \frac{|{\cal F}|}{6}$ cannot happen.
The cases for $v_2 \in \{v\}$ or $v_2 \in V(F^i)$ for $i \in \{3, \ldots, \omega(F)\}$ are considered in line~\ref{lin:v2F3}.
The case $V({\cal H}_{v_2,v_2}) = V(F^{\underline{1}}) \setminus V(F^{\underline{1}}_{f^{\underline{1}},f^{\underline{1}}})$ is considered in line~\ref{lin:G11comp1}. Then, we can assume that ${\cal H}_{v_2,v_2} = \top{B}{\underline{1}}$ where ${\cal B} = {\cal F}^{\underline{1}} - \top{F}{\underline{1}}$.
The cases for $v_3 \in \{v\}$ or $v_3 \in V(F^i)$ for $i \in \{3, \ldots, \omega(F)\}$ are considered in line~\ref{lin:v3F3}. Since we have already shown that $\{h,v_2,v_3\} \not\subseteq V(F^{\underline{1}})$, we can assume that $v_3 \in V(F^{\underline{2}})$. The possibilities for this case are considered in lines~\ref{lin:G1G1G2} and~\ref{lin:G1C1G2}.

\bigskip \noindent {\bf Case 2:} $|{\cal F}^{\underline{1}}| \ge \frac{2|{\cal F}|}{3}$ (lines~\ref{lin:begin23} to~\ref{lin:best}).

From now on, when we refer to ${\cal R}, {\cal B}$ and $\Gamma^1$ without mentioning the line of the algorithm we are considering, assume that their values are the instances of the homonym variables of the algorithm when the while loop finishes. The conditions of the while loop (line~\ref{lin:while}) and $|{\cal F}^{\underline{1}}| \ge \frac{2|{\cal F}|}{3}$ imply that $V(R) \subseteq V(F^{\underline{1}})$.
We need some claims.

The first claim deals with the case where $\overline{r}$ is ancestral of $v$ in ${\cal H}$ and has two ancestrals in ${\cal H}$ that are not ancestrals of $\overline{r}$ in ${\cal F}$.

\begin{claim} \label{cla:pivot}
If there are $w_1, w_2 \not\in [f^{\underline{1}}, \overline{r}]_{{\cal F}^{\underline{1}}}$ such that
$w_1 \in [h,w_2)_{\cal H}$ and
$w_2 \in (w_1,\overline{r})_{\cal H}$, then
there are $x_1,x_2 \in [r,\overline{r})_{\cal R}$ such that
${\cal H}_{h,w_2} = \langle {\cal R}_{r,\pi(x_1)}, W_1, {\cal R}_{x_1,x_2}, W_2\rangle$,
$W_1 \in \Upsilon_2(G - {\cal H}_{h,\pi(x_1)},v)$,
$W_2 \in \Upsilon_2(G - {\cal H}_{h,x_2},v)$ and
$|\langle {\cal R}_{r,\pi(x_1)}, W_1\rangle| \ge \frac{|{\cal F}|}{3}$.
\end{claim}

\begin{proofClaim}{cla:pivot}
Corollary~\ref{cor:minTPC} and the construction of ${\cal R}$ (lines~\ref{lin:begin23} and~\ref{lin:F1}) imply that ${\cal H}_{h,w_2} = \langle {\cal R}_{r, \pi(x_1)}$, $W_1, {\cal R}_{x_1,x_2}, W_2\rangle$,
$W_1 \in \Upsilon_2(G - {\cal H}_{h,\pi(x_1)},v)$ and
$W_2 \in \Upsilon_2(G - {\cal H}_{h,x_2},v)$.

Now, suppose by contradiction that $|\langle {\cal R}_{r,\pi(x_1)} , W_1\rangle| < \frac{|{\cal F}|}{3}$.
If $w_1 = h$, then $w_1 \not\in V(F^{\underline{1}})$ because 
$W_1 \in \Upsilon_2(G - {\cal H}_{h,\pi(x_1)},v)$
and ${\cal H}_{h,\pi(x_1)} = \varnothing$, which means that
$H$ has at least $\frac{2|{\cal F}|}{3}$ fill edges each with one extreme in $w_1$ because $|F^{\underline{1}}| \ge \frac{2|{\cal F}|}{3}$. If $w_1 \ne h$, then $H$ has at least $\frac{2|{\cal F}|}{3}$ fill edges each with one extreme in $w_1$ or $h = f^{\underline{1}}$ because $|\langle {\cal R}_{r,x_1} , W_1\rangle| < \frac{|{\cal F}|}{3}$ and every vertex of ${\cal H}_{x_1}$ has in $G$ one of these two vertices as a non-neighbor. Next, note that $|{\cal F}^{\underline{1}}_{\overline{r}}| \ge \frac{|{\cal F}|}{3}$. Since every vertex of ${\cal F}^{\underline{1}}_{\overline{r}}$ has a fill edge to $w_2$, there are at least $\frac{|{\cal F}|}{3}$ fill edges with both extremes in ${\cal H}_{w_2}$, which means that $fill(H) \ge |{\cal F}| > fill(U(G,v))$, a contradiction.
\end{proofClaim}

The second lemma says that for all but to at most two trees added to $\Gamma^1$ in line~\ref{lin:RF2}, a minimum TP completion is given by the universal TP completion.

\begin{claim} \label{cla:2ormore}
Denote by ${\cal R}^{\ell}, {\cal B}^{\ell}$ and $\Upsilon^{\ell}$ the instances of ${\cal R}, {\cal B}^{\underline{2}}$ and $\Upsilon_2$, respectively, in line~$\ref{lin:Y2}$ of the ${\ell}$-th iteration of the while loop.
If there are positive integers $i < j < k$ such that 
$|{\cal B}^i| < |{\cal B}^j| < |{\cal B}^k|$,
${\cal M}^{\ell} = \langle {\cal R}^{\ell}, {\cal Y}^{\ell} \rangle\in \Gamma^1$ for $\ell \in \{i,j,k\}$ where
${\cal Y}^{\ell} \in \Upsilon^{\ell}$
and ${\cal H} = \langle {\cal M}^i, {\cal P} \rangle$, then ${\cal P} = \langle {\cal R}_{r',r''}, U(G- (V({\cal M}^i) \cup V({\cal R}_{r',r''})) \rangle$ where
$\pi_{\cal R}(r') = \overline{r_i}$ and
$r'' \in [\overline{r_i},\pi_{{\cal F}^{\underline{1}}}(b^j)]_{\cal R}$.
\end{claim}

\begin{proofClaim}{cla:2ormore}
We begin remarking that the choice	$r'' =\overline{r_i}$ corresponds to the case where ${\cal R}_{r',r''} = \varnothing$. By the construction, we know that
${\cal R}^i \subset {\cal R}^j \subset {\cal R}^k$ and
$V({\cal Y}^{\ell}) \subseteq V({\cal B}^{\ell})$ for $\ell \in \{i,j,k\}$, that
$\pi_{{\cal F}^{\underline{1}}}(b^j) \in (\overline{r_i}, \pi_{{\cal F}^{\underline{1}}}(b^k))_{\cal R}$, that
$\pi_{{\cal F}^{\underline{1}}}(b^k) \in (\pi_{{\cal F}^{\underline{1}}} (b^j), \overline{r})_{\cal R}$, and that
${\cal B}^j$ and ${\cal B}^k$ have disjoint vertex subsets of $V({\cal F}^{\underline{1}})$ both containing vertices of $S$, see Figure~\ref{fig:R}.

Suppose by contradiction that the claim is false, i.e.,
${\cal P} \ne \langle {\cal R}_{r',r''}, U(G- (V({\cal M}^i \cup V({\cal R}_{r',r''})) \rangle$
for any choice of $r''$ in $[\overline{r_i},\pi_{{\cal F}^{\underline{1}}}(b^j)]_{\cal R}$.
Applying Lemma~\ref{lem:basic}~$(\ref{ite:F2root})$ on ${\cal H}_{y^i}$, we conclude that the number of fill edges with both extremes in ${\cal H}_p = {\cal P}$ is less than $|{\cal B}^i|$.
On the one hand, ${\cal P} = \langle {\cal R}_{r',\pi(r^*)}$, ${\cal S}', {\cal K}' \rangle$, where ${\cal S}'$ is a minimum slim TP completion of $(F^{\underline{1}}_{r^*}, N_G(v) \cap V(F^{\underline{1}}_{r^*}))$ for some $r^* \in [r',r'']_{\cal R}$ and ${\cal K}'$ is a minimum TP completion of $G - (V({\cal M}^i) \cup V({\cal R}_{r',\pi(r^*)}) \cup V({\cal S}'))$.
We have a contradiction in the case because the construction of ${\cal S}'$ needs of at least $|{\cal B}^j|$ fill edges and $|{\cal B}^j| > |{\cal B}^i|$. 
On the other hand, ${\cal P} = \langle {\cal R}_{r',r''}, {\cal K}'' \rangle$
where ${\cal K}''$ is a minimum TP completion of $G - (V({\cal M}^i) \cup V({\cal R}_{r',r''})$.
Lemma~\ref{lem:basic}~$(\ref{ite:boundF2})$ leads to a contradiction in this case because the second greatest tree of $G - ({\cal M}^i \cup {\cal R}_{r',r''})$ is ${\cal B}^j$ and $|{\cal B}^j| \ge |{\cal B}^i|$.
\end{proofClaim}

\begin{figure}[h]
\begin{center}

\usetikzlibrary{shapes.geometric}

\begin{tikzpicture}[scale=1]

\pgfsetlinewidth{1pt}

\tikzset{
	vertex/.style={circle,  draw, minimum size=5pt, inner sep=1pt}}

\tikzset{
	bigVertex/.style={circle,  draw, minimum size=0.5cm, inner sep=1pt}}

\node [vertex] (h) at (4,5) [label=right:$f^{\underline{1}}$]{};
\node [vertex] (u1) at (4,4) [label=right:$b^{\underline{1}}_{(2)}$]{} edge (h);
\node [vertex] (u2) at (4,3) [label=right:$b^{\underline{1}}_{(3)}$]{} edge (u1);
\node [vertex] (u3) at (4,2) [label=right:$b^{\underline{1}}_{(4)}$]{} edge (u2);
\node [vertex] (u4) at (4,1) [label=left:$b^{\underline{1}}_{(5)}$]{} edge (u3);
\node [vertex] (r) at (4,0) [label=left:$\overline{r}$]{} edge (u4);

\node [bigVertex] (x) at (3,4) [dashed, label=right:$$]{} edge (h);
\node [bigVertex] (x) at (3,3) [dashed, label=right:$$]{} edge (u1);
\node [bigVertex] (x) at (6,-3) [minimum size=3.2cm, label=left:$$]{ } edge (r);
\node [bigVertex] (x) at (3,-3) [dashed, minimum size=2cm, label=left:$$]{ } edge (r);

\node [bigVertex] (F3) at (16,3) [minimum size=1cm, label=above:${\cal F}^{\underline{3}}$] {};

\node [bigVertex] (F2) at  (12.5,2.5) [minimum size=1.8cm, label=above:${{\cal F}^{\underline{2}} = {\cal B}_{(1)}^{\underline{2}} = {\cal B}_{(2)}^{\underline{2}}}$] {};

\node [bigVertex] (B22) at  (8,2) [minimum size=2.6cm, label=above:${{\cal B}_{(3)}^{\underline{2}} = {\cal B}_{(4)}^{\underline{2}}}$] {};

\draw (u1) to (B22);

\node [bigVertex] (B24) at  (10,-2) [minimum size=3.2cm, label=right:${{\cal B}_{(5)}^{\underline{2}} = {\cal B}_{(6)}^{\underline{2}}}$] {};

\draw (u3) to (B24);

\end{tikzpicture}

\caption{Dashed circles mean subtrees containing no vertices of $S$.
For every instance with $\omega({\cal F}) \ge 2$, it holds ${\cal F}^{\underline{2}} = {\cal B}_{(1)}^{\underline{2}}$.
In this instance, $\omega({\cal F}) = 3$. We also have ${\cal F}^{\underline{2}} = {\cal B}_{(2)}^{\underline{2}}$ because ${\cal F}^{\underline{2}}$ is the second greatest tree of $G - \top{F}{1}$. However, ${\cal F}^{\underline{2}} \ne {\cal B}_{(3)}^{\underline{2}}$ because ${\cal F}^{\underline{2}}$ is not the second greatest tree of $G - {\cal F}^{\underline{1}}_{f^{\underline{1}},u_1}$. The number of iterations of the while loop is 6 which is the level of $\overline{r}$ in ${\cal F}^{\underline{1}}$.}
\label{fig:R}
\end{center}
\end{figure}

\begin{claim} \label{cla:atmost2}
Only the two smallest sizes of trees ${\cal M}$ added to family $\Gamma^2$ in line~$\ref{lin:saveComp}$ can satisfy ${\cal H} = \langle {\cal M}, {\cal K} \rangle$ where ${\cal K}$ is a minimum TP completion of $G - V({\cal M})$.
\end{claim}

\begin{proofClaim}{cla:atmost2}
Suppose that ${\cal M}^1, {\cal M}^2, {\cal M}^3 \in \Gamma^2$ when line~$\ref{lin:saveComp}$ is reached with $|{\cal M}^1| > |{\cal M}^2| > |{\cal M}^3|$. Note that ${\cal M}^i$ was added to $\Gamma^2$ before ${\cal M}^{i+1}$ for $i \in [2]$.
For $i \in [3]$, write
${\cal M}^i = \langle {\cal R}^i, {\cal T}^i \rangle$
where ${\cal T}^i$ is the slim TP completion constructed in line~\ref{lin:compF1} and
${\cal H}^i = \langle {\cal M}^i, {\cal K}^i \rangle$ where
${\cal K}^i$ is a minimum TP completion of $G - V({\cal M}^i)$. Denote by ${\cal W}^i$ the subtree of ${\cal F}^{\underline{1}}$ induced by the vertices of ${\cal T}^i$.
We know that ${\cal R}^1 \subset {\cal R}^2 \subset {\cal R}^3$, that ${\cal T}^3 \subset {\cal T}^2 \subset {\cal T}^1$ and that ${\cal P}^1 \subset {\cal P}^2 \subset {\cal P}^3$ where ${\cal P}^i = {\cal F}^{\underline{1}} - ({\cal R}^i \cup {\cal T}^i)$, see Figure~\ref{fig:R}.

Since ${\cal R}^2 \cup {\cal T}^2 \subset {\cal R}^1 \cup {\cal T}^1$ and ${\cal R}^1 \subset {\cal R}^2$, there exists $s_1 \in S \cap V({\cal T}_1 - ({\cal R}_2 \cup {\cal T}_2))$. Analogously, there exist $s_2 \in V({\cal T}_2 - ({\cal R}_3 \cup {\cal T}_3)) \cap S$ and $s_3 \in V({\cal T}_3) \cap S$. Note that $\{s_1,s_2,s_3\}$ is an independent set of $G$.

Let $w_a$ be the vertex of ${\cal W}^1$ with minimum level such that $s_1 \in V({\cal W}^1_{w_a})$ and $s_2 \not\in V({\cal W}^1_{w_a})$; and let $w_b$ be the vertex of ${\cal W}^1$ with minimum level such that $s_2 \in V({\cal W}^1_{w_b})$ and $s_1 \not\in V({\cal W}^1_{w_b})$.
Recall that ${\cal T}^1$ is a minimum slim TP completion of $(W^1, N_G(v) \cap V(W^1))$. Therefore, $s_1$ and $s_2$ are universal vertices in ${\cal T}^1$. Since $w_a$ and $w_b$ are siblings and $V({\cal W}^2) \subseteq V({\cal W}^1_{w_b})$, it holds that there exist at least $|{\cal T}^2|$ fill edges in $H$ each one joining $s_1$ to the vertices of ${\cal W}^1_{w_b}$.

Let $w_c$ be the vertex of ${\cal W}^2$ with minimum level such that $s_2 \in V({\cal W}^2_{w_c})$ and $s_3 \not\in V({\cal W}^2_{w_c})$; and let $w_d$ be the vertex of ${\cal W}^2$ with minimum level such that $s_3 \in V({\cal W}^2_{w_d})$ and $s_2 \not\in V({\cal W}^2_{w_d})$.
Recall that ${\cal T}^2$ is a minimum slim TP completion of $(W^2, N_G(v) \cap V(W^2))$. Therefore, $s_2$ and $s_3$ are universal vertices in ${\cal T}^2$.
Since $w_c$ and $w_d$ are siblings and $V({\cal W}^3) \subseteq V({\cal W}^2_{w_d})$, it holds that there exist at least $|{\cal T}^3|$ fill edges in $H$ each one joining $s_2$ to the vertices of ${\cal W}^2_{w_d}$.

Note that every vertex $u \in V({\cal P}^3)$ has the three fill edges $us_1,us_2,us_3$ in ${\cal H}^1$. Hence, we have that the number of fill edges joining vertices of ${\cal F}^{\underline{1}}$ is at least $|{\cal T}^2| + |{\cal T}^3| + 3|{\cal P}^2|$. 
By the conditions of the while loop (line~\ref{lin:while}) and the fact that $|{\cal F}^{\underline{1}}| \ge \frac{2|{\cal F}|}{3}$, we have that $|{\cal R}^i| < \frac{|{\cal F}^{\underline{1}}|}{2}$, which means that $|{\cal T}^i| + |{\cal P}^i| \ge \frac{|{\cal F}^{\underline{1}}|}{2}$ for $i \in [3]$.
Since $|{\cal P}^3| > |{\cal P}^2| > |{\cal P}^1|$, we have that $|{\cal T}^2| + |{\cal T}^3| + 3|{\cal P}^2| \ge |{\cal F}^{\underline{1}}|$.
Since every vertex not in $F^{\underline{1}}$ has a fill edge to $h$, it holds that $fill(H) > |V(F)| > fill(U(G,v))$, a contradiction.
\end{proofClaim}

As in the case where $|{\cal F}^{\underline{1}}| < \frac{2|{\cal F}|}{3}$, some rooted trees are saved in the family $\Gamma$ (lines~\ref{lin:saveComp} and~\ref{lin:saveR}). Then, in line~\ref{lin:recR}, for each ${\cal T}$ added to $\Gamma$, a recursive call is done obtaining as result a tree ${\cal T}'$. Then, the TP completion $\langle {\cal T} , {\cal T}' \rangle$ is saved in ${\cal Q}$ as a candidate to be a minimum TP completion of $G$. Other TP completions of $G$ are saved in ${\cal Q}$ in lines~\ref{lin:RF2} and~\ref{lin:RU} and in the call to {\sc FindTop} (line~\ref{lin:v}). In line~\ref{lin:best}, the TP completion with minimum number of edges contained in ${\cal Q}$ is chosen as a minimum TP completion of $G$. Therefore, we have to show that ${\cal H}$ is added to ${\cal Q}$ in one of these lines.

By Corollary~\ref{cor:minTPC}, we can write ${\cal H} = \langle {\cal X}^1 , \ldots , {\cal X}^k \rangle_\Lambda$ for $k \ge 1$. Let $p \in [k]$ be such that $\overline{r} \in V({\cal X}^p)$. We consider 3 subcases:

\bigskip \noindent {\bf Case 2.1:} $V(X^i) \subseteq V(F^{\underline{1}})$ for every $i \in [p-1]$.

First, consider that ${\cal X}^p = {\cal R}_{\overline{r}}$.
The case where $|{\cal R}| \ge \frac{|{\cal F}|}{3}$ is considered in line~\ref{lin:saveR}.
If $|{\cal R}| < \frac{|{\cal F}|}{3}$, then $|{\cal B}^{\underline{1}}| < \frac{|{\cal B}|}{3}$, which means that if ${\cal H} = \langle {\cal R}, {\cal K} \rangle$ where ${\cal K}$ is a minimum TP completion of $G - V({\cal R})$, then Lemma~\ref{lem:basic}~$(\ref{ite:oneThird})$ implies that $\ell_{\cal K}(v) \le 2$. This case is covered in line~\ref{lin:RU}.

Now, consider that ${\cal X}^p \ne {\cal R}_{\overline{r}}$, i.e., ${\cal X}_p$ is a minimum slim TP completion ${\cal T}$ of ${\cal F}^{\underline{1}}_{r'}$ for some $r' \in [r,\overline{r}]_{\cal R} \setminus S$. This case is considered in lines~\ref{lin:f1S} to~\ref{lin:saveComp}.
Because of the if condition of line~\ref{lin:save13}, we have to prove that we only need to save $\langle {\cal R}, {\cal T} \rangle$ in line~\ref{lin:saveComp} if $\frac{|{\cal R}|}{3} + |{\cal T}| \ge \frac{|{\cal F}|}{3}$. Suppose by contradiction that $|{\cal T}| < \frac{|{\cal F}| - |{\cal R}|}{3}$ and write ${\cal H} = \langle {\cal R}, {\cal W} \rangle$. Observe that the vertices of ${\cal T}$ form the biggest tree of ${\cal F} - {\cal R}$, which implies, by items~$(\ref{ite:oneThird})$ and~$(\ref{ite:hereditary})$ of Lemma~\ref{lem:basic},
that $\ell_{\cal W}(v) \le 2$. However, these cases have already been considered in line~\ref{lin:v}. Then, we can assume that $|{\cal T}| \ge \frac{|{\cal F}| - |{\cal R}|}{3}$ or equivalently that $\frac{|{\cal R}|}{3} + |{\cal T}| \ge \frac{|{\cal F}|}{3}$.

According to Claim~\ref{cla:atmost2}, we can discard the trees added to $\Gamma^2$ in line~\ref{lin:saveComp} having more vertices than the two smallest sizes. Now, consider two trees ${\cal T}$ and ${\cal T}'$ added to $\Gamma^2$ in line~\ref{lin:saveComp} such that $V({\cal T}) = V({\cal T}')$. Note that if $|T| > |T'|$, then
${\cal H}$ cannot be $\langle {\cal T}, {\cal T}'' \rangle$, because $\langle {\cal T}', {\cal T}'' \rangle$ has less fill edges than $\langle {\cal T}, {\cal T}'' \rangle$. The trees that cannot be part of a minimum solution by these observations are eliminated in lines~\ref{lin:G3aa} and~\ref{lin:G3b}.

\bigskip \noindent {\bf Case 2.2:} There is exactly one $i \in [p-1]$ such that ${\cal X}^i \in \Upsilon_2(G - {\cal H}_{h,\pi(x^i)},v)$.

We start this case proving that the proper ancestrals in ${\cal Z}$ of the vertex $b$ chosen in line~$\ref{lin:lastB2}$ do not need to be considered in the for loop beginning in line~\ref{lin:merge}. Let ${\cal Z}$, $b$ and ${\cal C} \in \Upsilon_2$ be chosen in lines~$\ref{lin:bestTop}$,~$\ref{lin:lastB2}$ and~$\ref{lin:T2}$ of a same iteration of the for loop beginning in line~$\ref{lin:bestTop}$, respectively. Therefore, it suffices to show that for any $z' \in [z,\pi_{{\cal F}}(b))$, it holds $|E(\langle {\cal Z}_{z,\pi(z')} , {\cal P} \rangle| \ge |E(U(G,v))|$ where $P$ is a minimum TP completion of $(G - V({\cal Z}_{z,\pi(z')}), S \cap V({\cal Z}_{z,\pi(z')}))$.
Write ${\cal B} = {\cal F} - {\cal Z}$ and ${\cal D} = {\cal F} - {\cal Z}_{z,z'}$. Note that $b = b^{\underline{1}}$ and that $|V(D^{\underline{2}})| < |V(B^{\underline{1}})|$. By Lemma~\ref{lem:basic}~$(\ref{ite:F2root})$, the number of fill edges of $\langle {\cal Z}_{z,\pi(z')} , {\cal P} \rangle$ that are not incident to $z'$ is smaller than $|V(D^{\underline{2}})|$. However, since there is exactly one $i \in [p-1]$ such that ${\cal X}^i \in \Upsilon_2(G - {\cal H}_{h,\pi(x^i)},v)$, there are $|V(B^{\underline{1}})|$ fill edges in $\langle {\cal Z}_{z,\pi(z')} , {\cal P} \rangle$ that are not incident in $z'$, a contradiction. Hence, in line~\ref{lin:merge}, we only need to consider the vertices $z^*$ belonging to $[\pi_{{\cal F}}(b), \overline{z}]_{\cal Z}$.

On the one hand, ${\cal X}_p = {\cal R}_{\overline{r}}$.
Then, we can write ${\cal H}_{\overline{r}} = \langle {\cal R}_{\overline{r}}, {\cal K} \rangle$ where ${\cal K}$ is a minimum TP completion of $G - V({\cal H}_{h,{\overline{r}}})$.
If $|{\cal R}| < \frac{|{\cal F}|}{3}$, then $|{\cal B}^{\underline{1}}| < \frac{|{\cal B}|}{3}$. Hence, by Lemma~\ref{lem:basic}~$(\ref{ite:oneThird})$, it holds that $\ell_{\cal K}(v) \le 2$. This case is covered in lines~\ref{lin:T2a} to~\ref{lin:QF2a}.
The case where $|{\cal R}| \ge \frac{|{\cal F}|}{3}$ is considered in lines~\ref{lin:lastB2} to~\ref{lin:QF2} using the tree ${\cal R}$ saved in line~\ref{lin:saveR}.

On the other hand, ${\cal X}_p \ne {\cal R}_{\overline{r}}$, i.e., ${\cal X}_p$ is a slim TP completion ${\cal T}$ of ${\cal F}^{\underline{1}}_{r'}$ for some $r' \in [r,\overline{r}]_{\cal R} \setminus S$. This case is considered in lines~\ref{lin:lastB2} to~\ref{lin:QF2} using the tree $\langle {\cal R}, {\cal T} \rangle$ saved in $\Gamma^2$ in line~\ref{lin:saveComp}.

\bigskip \noindent {\bf Case 2.3:} There are at least two $i \in [p-1]$ such that $X^i \in \Upsilon_2(G - {\cal H}_{h,\pi(x^i)},v)$.

Therefore, there are $y,y' \in [r,\overline{r}]_{\cal R}$ and $j,\ell \in [p-1]$ such that
${\cal H}_{h,x^\ell} = \langle {\cal X}^1$, $\ldots, {\cal X}^\ell \rangle = \langle {\cal R}_{r,\pi(y)}, {\cal X}^j , {\cal R}_{y,y'}, {\cal X}^\ell \rangle$,
$X^j \in \Upsilon_2(G - {\cal R}_{r,\pi(y)},v)$ and
$X^\ell \in \Upsilon_2(G - ({\cal R}_{r,y'} \cup X^j),v)$.
By Claim~\ref{cla:pivot}, we know that $|\langle {\cal R}_{r,\pi(y)} , {\cal X}^j \rangle| \ge \frac{|{\cal F}|}{3}$.
Furthermore, we claim that we can discard $\langle {\cal R}_{r,\pi(y)}, {\cal X}^j \rangle$
if $\pi(x^j) \ne \pi(y)$ and $|\langle {\cal R}_{r,\pi(\pi(y))}, {\cal X}^j \rangle| \ge \frac{|{\cal F}|}{3}$. Indeed, 
by the algorithm, we know that $|{\cal R}_{r,\pi(y)}| < \frac{|{\cal F}|}{3}$. Since $|\langle {\cal R}_{r,\pi(\pi(y))}, {\cal X}^j \rangle| \ge \frac{|{\cal F}|}{3}$, we have that $|{\cal X}^j| > |{\cal H}_\pi(y)|$, which implies that 
$|\langle {\cal R}_{r,\pi(\pi(y))}, {\cal X}^j, \ldots,$ ${\cal X}^k \rangle| \le |\langle {\cal R}_{r,\pi(y)}, {\cal X}^j, \ldots, {\cal X}^k \rangle|$.
These constraints are guaranteed by the if conditions in line~\ref{lin:F2minlevel}.

Finally, every tree saved in $\Gamma^1$ in line~\ref{lin:RF2} in considered in the for loop beginning in line~\ref{lin:M}. If ${\cal M}^1$ satisfies the if conditions of line~\ref{lin:2parents}, then, by Claim~\ref{cla:2ormore}, the possibilities for ${\cal M}^1$ are covered in for loop beginning in line~\ref{lin:G1a}. Otherwise, a recursive call for ${\cal M}^i$ is made in line~\ref{lin:G1b}.
\end{proof}

\begin{theorem} \label{thm:complexity}
Algorithm~$\ref{alg:MinTPC}$ runs in $O(n^7)$ steps where $n$ is the order of the input graph.
\end{theorem}

\begin{proof}
The algorithm is divided into two cases, one is when $|{\cal F}^{\underline{1}}| < \frac{2|{\cal F}|}{3}$ and the other is when $|{\cal F}^{\underline{1}}| \ge \frac{2|{\cal F}|}{3}$. Then, we can express the number of steps of the algorithm by $T(n) = T_1(n) + T_2(n)$, where $T_1(n)$ is associated with the first case and $T_2(n)$ with the second case. The algorithm has five lines with recursive calls, namely, lines~\ref{lin:G1G2small} and~\ref{lin:rec23} for $|{\cal F}^{\underline{1}}| < \frac{2|{\cal F}|}{3}$, and lines~\ref{lin:G1b},~\ref{lin:recR}, and~\ref{lin:QF2} for $|{\cal F}^{\underline{1}}| \ge \frac{2|{\cal F}|}{3}$. 

For the case $|{\cal F}^{\underline{1}}| < \frac{2|{\cal F}|}{3}$, we will show that
$T_1(n) \le 2T_1(\frac{5n}{6}) + 6 T_1(\frac{2n}{3})$.
The input graph of each recursive call in line~\ref{lin:G1G2small} has order at most $\frac{5|{\cal G}|}{6}$ because it is checked in line~\ref{lin:if56} if $|{\cal L}| \ge \frac{|{\cal F}|}{6}$. We have to show that line~\ref{lin:G1G2small} is executed at most twice. The tree ${\cal L}$ considered in line~\ref{lin:G1G2small} is formed by one member of $\Upsilon_1$ and one member of $\Upsilon_2$. We know that 
$|\Upsilon_1| \le 2$ and $|\Upsilon_2| \le 2$. However, if $|\Upsilon_1| = 2$, then there is a tree in $\Upsilon_1$ with the same order as ${\cal F}^{\underline{1}}$. Then, for this tree, line~\ref{lin:G1G2big} is executed instead of line~\ref{lin:G1G2small}, which means that line~\ref{lin:G1G2small} is executed at most twice.
The input graph of each recursive call in line~\ref{lin:rec23} has order at most $\frac{2|{\cal F}|}{3}$ because it is checked in line~\ref{lin:if23} if $|{\cal C}| \ge \frac{|{\cal F}|}{3}$.
Line~\ref{lin:rec23} is executed at most 6 times because $|\Gamma| \le 6$ when this line is reached. This is true because
at most two trees are added to it in line~\ref{lin:G1G1G2} because $|\Upsilon_2| \le 2$,
at most two in line~\ref{lin:G1G2big} because at most one member of $\Upsilon_1$ has ${\cal F}^{\underline{1}}$ vertices,
at most one in line~\ref{lin:G11comp1} and
at most one in line~\ref{lin:G11comp2}.

For the case $|{\cal F}^{\underline{1}}| \ge \frac{2|{\cal F}|}{3}$, we show that $T_2(n) \le 13 T_2(\frac{2n}{3})$.
The order of the input graph of each recursive call in line~\ref{lin:G1b} is clearly at most $\frac{2|{\cal F}|}{3}$ because it is checked in the if conditions of line~\ref{lin:F2minlevel} if $|\langle {\cal R}, {\cal J}\rangle| \ge \frac{|{\cal F}|}{3}$. The order of the input graph of every recursive call in lines~\ref{lin:recR} and~\ref{lin:QF2} is at most $\frac{2|{\cal F}|}{3}$ because it is checked in the if conditions of lines~\ref{lin:save13} and~\ref{lin:testR}, respectively, if the tree has at least $\frac{|{\cal F}|}{3}$ before it be added to $\Gamma$.
It remains to show that the number of recursive calls for this case is at most~$13$.

Denoting by ${\cal B}_{(j)}$ the instance of ${\cal B}$ in line~$\ref{lin:Y2}$ of the $j$-th iteration of the while loop, it holds that $|{\cal B}^{\overline{2}}_{(i_1)}| < |{\cal B}^{\overline{2}}_{(i_2)}|$ if $i_1 < i_2$. Therefore, the if condition of line~\ref{lin:2parents} and the fact that $|\Upsilon_2| \le 2$ for any instance of $\Upsilon$ imply that the number of recursive calls in line~\ref{lin:G1b} is at most 4.

At most two trees added to family $\Gamma^2$ in line~\ref{lin:saveComp} are saved in $\Gamma$ in lines~\ref{lin:G3a} to~\ref{lin:G3b}.
Since one more tree can be added to $\Gamma$ in line~\ref{lin:saveR}, the family $\Gamma$ has at most~$3$ elements when line~\ref{lin:bestTop} is reached. For each member of $\Gamma$, at most~$3$ recursive calls are done in lines~\ref{lin:recR} and~\ref{lin:QF2}, which result in at most $9$ for this subcase. Then, the total number of recursive calls of this case is at most~$13$.

Using Lemma~\ref{lem:mincompaction}, it is easy to see that the remaining operations of the algorithm cost $O(n^5)$. Using the Master Theorem~\cite{Cormen}, we conclude that
$T_2(n) = O(n^{\log_{\frac{3}{2}} 13}) = O(n^7)$.
We use induction on $n$ to prove that $T_1(n) < n^7$ for $n$ sufficiently large.
$T_1(n) \le 6 T_1(\frac{2n}{3}) + 2 T_1(\frac{5n}{6}) + n^6 \le 6 (\frac{2n}{3})^7 + 2 (\frac{5n}{6})^7 + n^6 < \frac{91}{100}n^7 + n^6 < n^7$. Therefore $T(n) = O(n^7)$.
\end{proof}

\bibliographystyle{plain}

\end{document}